\documentclass{amsart}
\usepackage{amssymb,amsmath,mathrsfs,amsthm,amscd,esint,color}
\usepackage[letterpaper,left=2.5cm,top=2.5cm,bottom=2.5cm, right=2.5cm]{geometry}

\newtheorem{theorem}{Theorem}[section]
\newtheorem{lemma}[theorem]{Lemma}
\newtheorem{proposition}[theorem]{Proposition}
\newtheorem{corollary}[theorem]{Corollary}

\newtheorem{definition}[theorem]{Definition}

\theoremstyle{definition}
\newtheorem{example}[theorem]{Example}

\theoremstyle{remark}
\newtheorem{remark}[theorem]{Remark}

\newcommand{\R}{\mathbb{R}}
\newcommand{\Rn}{{\R^n}}

\newcommand{\N}{\mathbb{N}}

\newcommand{\X}{\mathbb{X}}

\newcommand{\eps}{\varepsilon}

\newcommand{\cG}{{\mathscr{G}}}
\newcommand{\cE}{{\mathscr{E}}}
\newcommand{\cD}{{\mathscr{D}}}

\newcommand{\cM}{{\mathcal{M}}}

\newcommand{\cB}{{\mathcal{B}}}

\newcommand{\Linfty}{{L^\infty}}
\newcommand{\Lp}{{L^p}}
\newcommand{\Lploc}{{L^p_{\text{loc}}}}
\newcommand{\Lone}{{L^1}}
\newcommand{\Loneloc}{{L^1_{\text{loc}}}}

\newcommand{\om}{\{X_{\omega}\}_{\omega\in\Omega}}
\newcommand{\BMO}{\mathrm{BMO}_{\cG}(\X)}
\newcommand{\BMOp}{\mathrm{BMO}^p_{\cG}(\X)}

\begin{document}

\title [BMO and the John-Nirenberg Inequality on Measure Spaces]
{BMO and the John-Nirenberg Inequality on Measure Spaces}

\author[Dafni]{Galia Dafni}
\address{(G.D.) Concordia University, Department of Mathematics and Statistics, Montr\'{e}al, QC H3G 1M8, Canada}
\curraddr{}
\email{galia.dafni@concordia.ca}
\thanks{G.D. was partially supported by the Natural Sciences and Engineering Research Council (NSERC) of Canada and the Centre de recherches math\'{e}matiques (CRM). R.G. was partially supported by the Centre de recherches math\'{e}matiques (CRM), the Institut des sciences math\'{e}matiques (ISM), and the Fonds de recherche du Qu\'{e}bec -- Nature et technologies (FRQNT). A.L. was partially supported by the Natural Sciences and Engineering Research Council (NSERC) of Canada, the Institut des sciences math\'{e}matiques (ISM), and a Concordia Undergraduate Research Award (CUSRA)}

\author[Gibara]{Ryan Gibara}
\address{(R.G.) Universit\'{e} Laval, D\'{e}partement de math\'{e}matiques et de statistique, Qu\'{e}bec, QC G1V 0A6, Canada}
\curraddr{}
\email{ryan.gibara@gmail.com}

\author[Lavigne]{Andrew Lavigne}
\address{(A.L.) McGill University, Department of Mathematics and Statistics, Montr\'{e}al, QC H3A 0B9, Canada}
\curraddr{}
\email{andrew77@me.com}


\date{}

\begin{abstract}
We study the space $\BMO$ in the general setting of a measure space $\X$ with a fixed collection $\cG$ of measurable sets of positive and finite measure, consisting of functions of bounded mean oscillation on sets in $\cG$. The aim is to see how much of the familiar BMO machinery holds when metric notions have been replaced by measure-theoretic ones. In particular, three aspects of BMO are considered: its properties as a Banach space, its relation with Muckenhoupt weights, and the John-Nirenberg inequality. We give necessary and sufficient conditions on a decomposable measure space $\X$ for $\BMO$ to be a Banach space modulo constants.  We also develop the notion of a Denjoy family $\cG$, which guarantees that functions in $\BMO$ satisfy the John-Nirenberg inequality on the elements of $\cG$.
\end{abstract}

\maketitle


\section{Introduction}

Functions of bounded mean oscillation were  introduced by John and Nirenberg in \cite{john2}, and shown to satisfy the by-now celebrated John-Nirenberg inequality.  Mean oscillation measures the average deviation of a locally integrable function from its mean on a certain set.  Classically, on $\Rn$, the mean oscillation is measured on all cubes $Q$ with sides parallel to the axes; sometimes the cubes are restricted to lie in a fixed cube. The space of functions of bounded mean oscillation, BMO, has proven useful  in harmonic analysis and partial differential equations as a replacement for $L^\infty$, as well as as the dual of the Hardy space $H^1$, and has important connections with the theory of quasi-conformal mappings. With the observation that an equivalent characterization of BMO is obtained by replacing cubes with balls, its reach has extended beyond the Euclidean setting, to manifolds \cite{mani} and metric measure spaces \cite{metric}.

This idea of replacing cubes as the sets over which mean oscillation is measured has lead various authors to consider more general BMO spaces, notable among which is the strong BMO space where cubes are replaced by rectangles \cite{cs,kor}. The different geometric properties of cubes versus rectangles manifest as different functional properties for the corresponding BMO spaces. In particular, the constants in various inequalities are affected by this change -- see \cite{kor2} for an overview of the cases of the boundedness of the decreasing rearrangement and the absolute value operator.

Replacing familiar geometries with more general bases is not a phenomenon limited to BMO. The theory of maximal functions with respect to a general basis is a well developed field, with ties to the theory of differentiation of the integral. The book of de Guzm\'{a}n \cite{guzman} is a standard reference in the area. Muckenhoupt weights have also been considered with respect to a general basis -- see, for instance, the work of Jawerth \cite{jawerth} and P\'{e}rez \cite{perez}, and more recent work \cite{duomo, ht, ns}, which also touches on the connection between $A_p$ weights and BMO in this general context.  Sometimes restrictive assumptions are imposed in order to obtain strong results, such as the John-Nirenberg inequality -- see for example \cite{buckley, hy}.

It is interesting to investigate what are the strongest results one could obtain about BMO under the weakest assumptions, thus isolating the features which are inherent to BMO. 
Motivated by this, in previous work \cite{dafni1}, two of the authors developed a theory of BMO spaces on a domain in $\R^n$ with Lebesgue measure, where the mean oscillation was taken over sets of finite and positive measure belonging to an open cover of the domain. In this work, the completeness of BMO was demonstrated (without any connection to duality) and sharp constants were studied for various inequalities. 

This motif is continued in the present paper, where BMO is developed on a measure space $\X=(X,\cM,\mu)$ with a fixed collection $\cG$ of measurable sets of positive and finite measure, which we call generators. We denote this space by $\BMOp$ (the $p$ is an integrability parameter -- see Definition \ref{defbmo}). The aim is to see how much of the familiar BMO machinery holds when metric notions have been replaced by measure-theoretic ones. In particular, three aspects of BMO are considered: its properties as a Banach space, its relation with Muckenhoupt weights, and the John-Nirenberg inequality. 

In Section 3, the definition of BMO as a Banach space is considered. When $\X$ is a decomposable measure space -- that is, a measure space equipped with a partition, not necessarily countable, consisting of sets of finite measure (see Definition~\ref{decomposable}) -- two properties of the collection of generators $\cG$ are identified as being both necessary and sufficient for the completeness of $\BMOp$:
\begin{theorem}\label{completeness}
For a complete decomposable measure space $\X$ and a collection of generators $\cG$, the following are equivalent:
\begin{enumerate}
\item[(i)] $(\X,\cG)$ has the finite chain property and is $\sigma$-decomposable,
\item[(ii)] $\|\cdot\|_{\BMOp}$ is a norm modulo constants,
\item[(iii)] $\BMOp$ is a Banach space modulo constants.
\end{enumerate}
\end{theorem}

The John-Nirenberg inequality, a fundamental property of the classical BMO space, is the statement that for any $G\in\cG$ and $f\in\BMO$, 
\begin{equation}\label{jnintro}
\mu(\{x\in G:|f(x)-f_G|>t  \} ) \leq c_1\mu(G)\exp\left(\frac{-c_2}{\|f\|_{\BMO}} t\right) \,,
\end{equation}
for all $t>0$ and some constants $c_1,c_2>0$, where $f_G$ is the mean of $f$ on $G$. This inequality does not hold in every BMO space, however (for example, see the space considered in \cite{cl,ly}). We say that $\BMO$ has the John-Nirenberg property if \eqref{jnintro} holds for all $G\in\cG$ and $f\in\BMO$, with constants $c$ and $C$ that are independent of $f$. 

In Section 4, a connection is established between the aforementioned theory of Muckenhoupt weights with respect to $\cG$, $A_2(\X,\cG)$, and the John-Nirenberg property -- see \cite{spanish} for the situation in the classical setting of cubes. It turns out that $\log A_2(\X,\cG)$ is always contained in $\BMO$, see Proposition \ref{loga2bmo}, but the converse need not hold. In Theorem \ref{jnpshs}, we show that having $\alpha f\in\log A_2(\X,\cG)$ for $f\in \BMO$,  for all sufficiently small $\alpha>0$,  is equivalent to $\BMO$ having the John-Nirenberg property.

In Section 5, a large class of $\BMO$ spaces is considered, where $\cG$ is a so-called Denjoy family -- see Definition \ref{defdenjoy}. At this level of generality, a decomposition theorem of Calder\'{o}n-Zygmund type is obtained, where measure-theoretic conditions on the basis $\cG$ have replaced any metric notions. This class of BMO spaces includes not only classical BMO, but also strong BMO and BMO on spaces of homogeneous type where $\cG$ is the collection of all open balls. We show that
\begin{theorem}
\label{Denjoytheorem}
If $\cG$ is a Denjoy family, then $\BMO$ has the John-Nirenberg property. 
\end{theorem}

Finally, in Section 6 we collect a list of examples that illustrate the applicability and subtleties of the theory developed in this paper. Extensive details are provided for the benefit of the reader. 


\section{Preliminaries}

\subsection{Bounded mean oscillation}

Let $\X=(X, \mathcal{M}, \mu)$ be a nontrivial semifinite measure space. Writing $\cM^{\ast} = \{E \in \cM: 0<\mu(E)< \infty\}$, assume that there exists a collection $\cG\subset\cM^{\ast}$ that covers $X$. The sets $G\in\cG$ will be called generators. 

We begin by defining a notion of local integrability with respect to this collection $\cG$. Here, and elsewhere throughout the paper, $1\leq{p}<\infty$. 

\begin{definition}\label{deflploc}
We say that a measurable function $f:X\rightarrow\R$ is locally $p$-integrable with respect to $\cG$ if 
$$
\int_G\!|f|^p\, d\mu < \infty
$$ 
for every $G \in \cG$. We denote the space of such functions by $\Lploc(\X,\cG)$. 
\end{definition}

This definition of local integrability is adapted to the choice of collection $\cG$ and is purely measure theoretic. When a topology is available, one often sees local integrability defined in terms of integrals over compact subsets. Example \ref{localintegrability} shows that the two notions need not agree. 

Generalizing the work done in \cite{dafni1}, we define a BMO space with respect to $\cG$. For a generator $G\in\cG$, we write
$$
f_G := \fint_G \! f \, d\mu := \frac{1}{\mu(G)} \int_G\! f \, d\mu\,.
$$

\begin{definition}\label{defbmo}
We say that a function $f\in \Lploc(\X,\cG)$ is in $\BMOp$ if
$$
\|f\|_{\BMOp}:=\sup\limits_{G \in \cG} \left( \fint_G\! |f-f_G|^p \, d\mu \right)^{1/p}\!<\infty\,.
$$
When $p=1$, we simply write $\BMO$.
\end{definition}

It is easy to see that $\|\cdot\|_{\BMOp}$ defines a seminorm, but cannot define a norm: if $f=K$ almost everywhere for some constant $K$ (note that such a function is in $\Lploc(\X,\cG)$ for any $1\leq p<\infty$ and for any choice of $\cG$), then $\|f\|_{\BMOp}=0$. If one considers the space $\BMOp$ modulo constants, then there is a question as to whether $\|\cdot\|_{\BMOp}$ defines a norm or not. 

The following section is devoted to investigating conditions on the pair $(\X,\cG)$ that guarantee an affirmative answer to this question. In the first instance, it will be necessary to specialize to a class of measure space that is more restrictive than just being semifinite. One can consider $\sigma$-finite measure spaces, but in the interest of obtaining results with the weakest possible assumptions, we consider what are known as decomposable measure spaces. 

\subsection{Decomposable measure spaces}

A partition of $X$ is a pairwise disjoint family $\om$ of nonempty sets in $\cM$ satisfying 
$$
X=\bigcup_{\omega\in\Omega}X_{\omega}\,.
$$
In general, partitions are not assumed to be countable, in the sense that the index set $\Omega$ need not be countable. 

\begin{definition}\label{decomposable}
A measure space $\X=(X, \cM, \mu )$ is said to be decomposable (or strictly localizable) if there is a partition $\om$ of $X$ such that
\begin{enumerate}
\item[(a)] $\mu(X_{\omega})<\infty$ for all $\omega\in\Omega$;
\item[(b)] if $A\subset{X}$ and $A \cap X_{\omega} \in \cM$ for all $\omega \in \Omega$, then $A \in \cM$; and,
\item[(c)] if $E \in \cM$, then 
$$
\mu(E) = \sum\limits_{\omega \in \Omega} \mu (E \cap X_{\omega} )\,.
$$
\end{enumerate}
Such a partition is called a decomposition of $\X$.
\end{definition}

Note that when the partition is uncountable, the sum in $(c)$ is interpreted as the supremum of all sums over finite subfamilies of the partition. 

The class of decomposable measure spaces includes most of the measure spaces seen in applications, and they form a natural intermediary between the $\sigma$-finite measure spaces and the semifinite measure spaces. Furthermore, they are precisely those measure spaces for which the canonical map between $\Linfty$ and the dual space $(\Lone)^*$ is an isometric isomorphism (see \cite[243F and 243G]{fremlin2}). 

An example of a decomposable measure space that is not $\sigma$-finite is $\R$ with its power set and the counting measure. In Example \ref{ex2}, a less standard example is constructed. When the counting measure on $\R$ is restricted to the $\sigma$-algebra of countable and co-countable sets, we obtain a semifinite measure space that is not decomposable.

Restricting to decomposable measure spaces allows us to make sense of the measure of some uncountable unions of sets. An example of this is the following.

\begin{proposition}\label{nullsets}
Let $\X$ be decomposable with decomposition $\om$. If $N^{\omega} \subset X_{\omega}$ is a null set for each $\omega \in \Omega$, then $N=\bigcup_{\omega \in \Omega} N^{\omega}$ is also a null set.
\end{proposition}

\begin{proof}
As $\om$ forms a partition of $X$, $N\cap X_{\omega}=N^{\omega}\in\cM$ for any $\omega\in\Omega$. From $(b)$ of the definition of decomposable measure space, it follows that $N\in\cM$. Then, $(c)$ implies that 
$$
\mu(N)= \sum\limits_{\omega \in \Omega} \mu (N \cap X_{\omega} ) = \sum\limits_{\omega \in \Omega} \mu(N^{\omega}) = 0\,.
$$
\end{proof}

\begin{corollary}\label{decomposable2}
Let $\X$ be decomposable with decomposition $\om$. Writing $\Omega_E = \{ \omega \in \Omega : \mu (E \cap X_{\omega}) > 0 \}$ for $E\in\cM$, we can redefine property $(c)$ in Definition \ref{decomposable} as follows:
$$
\mu(E) = \sum\limits_{\omega \in \Omega_E} \mu(E \cap X_{\omega})\,.
$$
As a consequence, if there are at most countably many $\{ X_k \}_{k=1}^{\infty} \subset \{ X_{\omega} \}_{\omega \in \Omega}$ such that $ \mu(E \cap X_k) > 0$, then 
$$
\mu(E) =  \sum\limits_{k=1}^{\infty} \mu(E \cap X_k)\,.
$$
\end{corollary}

Another useful fact is that $\Omega_{E}$ is countable whenever $E$ has finite measure. This follows from $(c)$ and the definition of an uncountable sum. We record the statement below for later reference.

\begin{proposition}\label{finitedecomposition}
Let $\X$ be decomposable with decomposition $\om$. If $E \in \cM$ has finite measure, then there exist at most countably many $\omega \in \Omega$ such that $\mu(E \cap X_{\omega})> 0$.
\end{proposition}


\section{$\BMOp$ as a Banach Space}

Consider a pair $(\X,\cG)$ comprised of a decomposable measure space $\X$ and a fixed collection of generators $\cG$. We assume that $\X$ is complete; otherwise, one may replace $\X$ with its completion. In this section, we give necessary and sufficient conditions on $(\X,\cG)$ to guarantee that $\|\cdot\|_{\BMOp}$ is a norm modulo constants, and that the resulting normed space $\BMOp$ is a Banach space.

\subsection{$\mathbf{\sigma}$-decomposability and the finite chain property}

The first property we consider is related to the notion of an {\em essential cover}, by which we mean a cover ignoring sets of measure zero. 

\begin{definition}\label{sigmadecomposable}
We say that $(\X,\cG)$ is $\sigma$-decomposable if there exists a decomposition $\om$ such that for every $\omega \in \Omega$, $\cG$ contains a countable essential subcover $\cG_{\omega} \subset \cG$ of $X_{\omega}$. The collection
$$
\bigcup_{\omega\in\Omega}\cG_{\omega}
$$
is called a $\sigma$-decomposition of $\X$.
\end{definition}

\begin{remark}\label{transfer}
Note that, by $(c)$ of Definition \ref{decomposable} and Proposition \ref{finitedecomposition}, if $(\X,\cG)$ is $\sigma$-decomposable with respect to some decomposition $\om$, then it will be $\sigma$-decomposable with respect to any other decomposition $\{X_\alpha\}_{\alpha\in\mathcal{A}}$.
\end{remark}

It follows that if $(\X,\cG)$ is $\sigma$-decomposable and, in addition, that $\X$ is $\sigma$-finite, then $\cG$ contains a countable essential subcover of $X$, and conversely. Thus, the notion of $\sigma$-decomposability is an analogue of the notion of Lindel\"{o}f from the topological setting (see Section 4.1) as it allows us to cover partitions of the space (or the whole space in the $\sigma$-finite case), up to a null set, with a countable subcollection of $\cG$.

\begin{proposition}\label{eomega}
If $(\X,\cG)$ is $\sigma$-decomposable, then there exists a decomposition $\cE=\{E_0\}\cup\{ E_{\omega} \}_{\omega \in \Omega}$ of $\X$ with the following properties:
\begin{enumerate}
\item[(i)] $E_0$ is a null set;
\item[(ii)] $E_{\omega}\in\cM^*$ for all $\omega\in{\Omega}$; and,
\item[(iii)] for every $\omega \in \Omega$, there exists a $G_{\omega} \in \cG$ such that $E_{\omega} \subset G_{\omega}$.
\end{enumerate}
We say that this decomposition is subordinate to $\cG$.
\end{proposition}

\begin{proof}
Let $\{ X_{\alpha} \}_{\alpha \in \mathcal{A}}$ be a decomposition of $\X$. By Remark \ref{transfer}, for every $\alpha \in \mathcal{A}$, there is a countable essential subcover $\cG_{\alpha}\subset\cG$ of $X_\alpha$. By the definition of essential cover, we can assume that $\mu (G \cap X_{\alpha}) > 0$ for every $G \in \cG_{\alpha}$.

First, we construct the set
$$
E_0 = \bigcup\limits_{\alpha \in \mathcal{A}} \left[ X_{\alpha} \backslash \left( \bigcup\limits_{G \in \cG_{\alpha}} G \right) \right]\,.
$$
From the definition of essential cover and Proposition \ref{nullsets}, it follows that $E_0$ is a null set, which is (i).

Fixing $\alpha \in \mathcal{A}$, write $\cG_{\alpha} = \{ G_n^{\alpha} \}$ and
$$
E^{\alpha}_n = X_{\alpha} \cap \left( G_n^{\alpha} \backslash \bigcup\limits_{k=1}^{n-1} G_k^{\alpha} \right)\,.
$$
By construction, $\{E^{\alpha}_n\}$ forms a partition of $X_\alpha\setminus E_0$ and $E^{\alpha}_n \subset G^{\alpha}_n$ for every $n\in\N$. This implies, in particular, that $\mu(E^{\alpha}_n)\leq \mu(G^{\alpha}_n)<\infty$ for every $n\in\N$. Furthermore, since the collection $\cG_{\alpha}$ is countable, we may assume that $\mu(E^{\alpha}_n)>0$ for every $n\in\N$. 

Doing this for every $\alpha\in\mathcal{A}$, we obtain a partition $\{E_n^\alpha\}_{n\in\N,\alpha\in\mathcal{A}}$ of $X\setminus E_0$. Renaming it $\{E_\omega\}_{\omega\in\Omega}$, the previous paragraph implies that (ii) and (iii) hold. It remains to show that $\cE=\{E_0\}\cup\{E_\omega\}_{\omega\in\Omega}$ is a decomposition of $\X$. From the construction, condition $(a)$ of Definition \ref{decomposable} is satisfied. 

Let $A\subset{X}$ be such that $A\cap E_\omega\in\cM$ for all $\omega\in\Omega$. Note that $A\cap E_0\in\cM$ is always true by completeness of $\X$. Fix an $\alpha\in\mathcal{A}$ and a countable subcollection $\{E_n^\alpha\}\subset\{E_\omega\}_{\omega\in\Omega}$ such that
$$
X_\alpha = (E_0\cap X_\alpha)\cup\left(\bigcup_n E_n^\alpha \right)\,.
$$
Hence,
$$
A\cap X_\alpha = (A \cap E_0\cap X_\alpha)\cup\left(\bigcup_n (A\cap E_n^\alpha) \right)\,.
$$
The set $A \cap E_0\cap X_\alpha$ is measurable by completeness and each $A\cap E_n^\alpha$ is measurable by assumption. As $\sigma$-algebras are closed under countable unions, $A\cap X_\alpha$ is measurable. Since this holds for every $\alpha\in\mathcal{A}$ and $\{ X_{\alpha} \}_{\alpha \in \mathcal{A}}$ is a decomposition of $\X$, it follows that $A$ is measurable. This is condition $(b)$ of Definition \ref{decomposable}.

Let $E \in \cM$. By Corollary \ref{decomposable2} and Proposition \ref{finitedecomposition}, if $\mu(E) < \infty$, there exists a countable subset $\{ X_k \} \subset \{ X_{\alpha} \}_{\alpha \in A}$ such that
$$
\mu(E) = \sum_{k} \mu(E \cap X_k )\,.
$$
By construction, for each $k$ there exists a countable subset $\{ E^k_{n} \} \subset \{ E_{\omega} \}_{\omega \in \Omega}$ that essentially covers $X_k$, and so 
$$
\mu(E)=\sum_{k} \mu(E \cap X_k) = \sum_{k,n} \mu(E \cap E^k_{n}) =  \sum\limits_{\omega \in \Omega} \mu(E \cap E_{\omega} )\,,
$$
where the last equality follows from the fact that $\mu(E\cap E_{\omega}) =0$ for every $E_{\omega} \notin \{E^k_{n}\}$.

If $\mu ( E ) = \infty$, then we have two subcases: If $E$ intersects at most countably many $X_{\alpha}$ of strictly positive measure, then the proof is similar to the finite case because of Corollary \ref{decomposable2}. If $E$ intersects uncountably many $X_{\alpha}$ of strictly positive measure, then it intersects uncountably many $E_{\omega}$ with measure greater than zero. So we get that $\sum_{\omega \in \Omega} \mu ( E \cap E_{\omega} )$ is a sum of uncountably many strictly positive numbers, and therefore equals infinity.

Combining all cases together, we get that for every $E \in \cM$,
$$
\mu ( E ) = \sum\limits_{\omega \in \Omega} \mu ( E \cap E_{\omega} )\,,
$$
which is $(c)$ of Definition \ref{decomposable}.
\end{proof}

From here on, when we write $\cE$ we mean a decomposition $\{E_0\}\cup\{ E_{\omega} \}_{\omega \in \Omega}$ of $\X$ subordinate to $\cG$. For a fixed $\omega \in \Omega $, we will reserve the notation $G_{\omega}$ as a fixed generator in $\cG$ such that $E_{\omega} \subset G_{\omega}$.

\begin{definition}\label{defFCP}
A finite chain is a finite subset of $\cG$, $\{G_n\}_{n=0}^{N}$, such that $\mu \left( G_{n-1} \cap G_n \right) > 0$ for all $1\leq{n}\leq{N}$. 

We say that $(\X,\cG)$ has the finite chain property (FCP) if for every $G$ and $G'$ in $\cG$, there exists a finite chain $\{G_n\}_{n=0}^{N}$ with $G=G_0$ and $G' = G_N$.
\end{definition}

Note that the connection of two sets by a finite chain is a reflexive, symmetric, and transitive property.

In \cite{dafni1}, the proof of completeness of BMO with respect to an open cover of a domain in $\Rn$ crucially makes use of path-connectedness. But in the end, it was because it implied the finite chain property that the result was obtained. Hence, in our general space, where we do not assume a topology, the FCP replaces path-connectedness and as we will see, along with $\sigma$-decomposability, not only plays a key role in the proof of completeness, but is actually a necessary condition.

\subsection{Proof of Theorem \ref{completeness}}

As pointed out in the preliminaries, the seminorm $\|\cdot\|_{\BMOp}$ cannot be a norm on $\BMOp$. It is, however, invariant under the map $f\mapsto f+K$ for any constant $K$, and so we consider $\BMOp$ modulo constants. As the kernel of $\|\cdot\|_{\BMOp}$ contains those functions that are constant almost everywhere, to show that $\|\cdot\|_{\BMOp}$ is a norm modulo constants, it suffices to show the following property:
\begin{equation}\label{normhomogeneity}
\|f\|_{\BMOp} = 0 \implies f = \mathrm{constant \; a.e. \; on} \; X\,.
\end{equation}

The goal of this subsection is to prove Theorem \ref{completeness}. We begin by presenting some lemmas that will be needed. 

\begin{lemma}\label{FCPconstant}
Let $f \in \BMOp$ satisfy $\|f\|_{\BMOp} = 0$ and $f = K$ a.e. on $G_0\in\cG$. If $\{ G_n \}_{n=0}^N$ is a finite chain, then $f = K$ a.e. on $G_N$.
\end{lemma}

\begin{proof}
For any $G\in\cG$, the assumption that $\|f\|_{\BMOp} = 0$ implies that
$$
\fint_G\! |f-f_G|^p \, d\mu = 0\,,
$$
which implies that $f = f_G$ almost everywhere on $G$. Applying this to $G_0$, it follows that $f=f_{G_0}$ almost everywhere on $G_0$. As the same is true with $K$ in place of $f_{G_0}$, by assumption, and $\mu(G_0)>0$, it must be that $f_{G_0}=K$. 

Applying this to $G_{k-1}$ and $G_{k}$, $1\leq k\leq N$, it follows that $f=f_{G_{k-1}}$ almost everywhere on $G_{k-1}$ and $f=f_{G_{k}}$ almost everywhere on $G_{k}$. On the intersection $G_{k-1}\cap G_{k}$, which has strictly positive measure by the definition of a finite chain, both $f=f_{G_{k-1}}$ and $f=f_{G_{k}}$ must hold almost everywhere. Hence, $f_{G_{k-1}}=f_{G_k}$. 

As this is true for all $1\leq k\leq N$, it follows that $f_{G_N}=f_{G_{N-1}}=\ldots=f_{G_0}=K$, and so $f=K$ almost everywhere on $G_N$. 
\end{proof}

\begin{definition}\label{deffinitelyconnected}
Let $\om$ be a decomposition of $\X$. We say that a pair $(X, \widetilde{X})$ from $\om$ is finitely connected if there exists a finite subset $\{ X_n \}_{n=0}^N \subset \om$ such that
\begin{itemize}
\item $X_0=X$, $X_N = \widetilde{X}$, and
\item for all $1\leq n\leq N$, there exists a $G_n \in \cG$ such that $\mu ( G_n \cap X_n ) > 0$ and $\mu ( G_n \cap X_{n-1} ) > 0$.
\end{itemize}
We say that $\om$ is pairwise finitely connected if every pair $(X, \widetilde{X})$ from $\om$ such that $\mu (X) > 0$ and $\mu (\widetilde{X}) > 0$ is finitely connected.
\end{definition}

This definition is analogous to the FCP, but at the level of the decomposition $\om$. This property is also symmetric and transitive, and it tells us that for each $(X, \widetilde{X})$ from $\om$ with $\mu (X) > 0$ and $\mu (\widetilde{X}) > 0$, one can build a finite chain between them, where each $X_n$ is ``attached'' to the next by a generator $G_n$. One can imagine the sets $\{ X_n \}_{n=0}^N$ as sheets of paper and the sets $\{ G_n \}_{n=1}^N$ as pieces of tape. When two sheets of paper are finitely connected, it means that we can tape the first to the second, the second to the third, and so on until we get to the last one.

\begin{lemma}\label{lemmanotnorm}
If there exists a set $E \in \mathcal{M}$ such that $\mu ( E ) > 0$ and $\mu (G \cap E ) = 0$ for every $G \in \cG$, then $\|\cdot\|_{\BMOp}$ is not a norm modulo constants on $\BMOp$.
\end{lemma}

\begin{proof}
Fix such a measurable set $E$ and let $f = \chi_E$. Note that $\mu (E^c) > 0$; otherwise, $\mu (G \cap E ) = 0$ would imply that $\mu (G ) = 0$ for every $G \in \cG$. As such, $f$ does not equal a constant almost everywhere on $X$. However, for every $G\in\cG$,
$$
\fint_G \! |f-f_G|^p \, d\mu =  \frac{\mu (G \cap E )\mu (G \cap E^c )^p+\mu(G\cap E)^p\mu(G\cap E^c)}{\mu(G)^{p+1}}=0\,,
$$
implying that $\|f\|_{\BMOp} = 0$. This shows that \eqref{normhomogeneity} does not hold and so $\|\cdot\|_{\BMOp}$ is not a norm modulo constants on $\BMOp$.
\end{proof}

\begin{lemma}\label{decompositionsfinitelyconn}
If $\|\cdot\|_{\BMOp}$ is a norm modulo constants on $\BMOp$, then every decomposition of $\X$ into sets of strictly positive measure is pairwise finitely connected.
\end{lemma}

\begin{proof}
Assume that there exists a decomposition $\{ X_{\omega} \}_{\omega \in \Omega}$ of $\X$ into sets of strictly positive measure for which there is a pair $(X_a, X_b)$ that is not finitely connected.

Let $\Omega_{a} = \{ \omega \in \Omega : X_{\omega} \; \mathrm{is \; finitely \; connected \; to} \; X_a \}$ and write 
$$
A = \bigcup\limits_{\omega \in \Omega_a} X_{\omega}\,.
$$
Consider also $\Omega_a^c=\Omega\setminus \Omega_a$ and $A^c=X\setminus A$. 

By Lemma \ref{lemmanotnorm}, since $\mu(X_a)>0$, there exists a $G\in\cG$ such that $\mu(G\cap X_a)>0$. This shows that, under these conditions, finite connectivity is reflexive; i.e.\ $X_a$ is finitely connected to itself. It follows that $\Omega_a$ is nonempty and so $\mu(A)\geq \mu(X_a)>0$. On the other hand, our assumption that $X_b$ cannot be finitely connected to $X_a$ implies that $\Omega_a^c$ is nonempty, as well, and so $\mu(A^c)\geq \mu(X_b)>0$. 

Assume now that there exists a $G \in \cG$ such that $\mu ( G \cap A ) > 0$ and $\mu ( G \cap A^c ) > 0$. Then, by Proposition \ref{nullsets}, there must exist some $\alpha \in \Omega_a$ and some $\beta \in \Omega_a^c$ such that $\mu (G \cap X_{\alpha} ) >0$ and $\mu ( G \cap X_{\beta} ) > 0$. This implies that $\beta \in \Omega_a$, which is a contradiction.

Thus, either $\mu (G \cap A ) = 0$ or $\mu ( G \cap A^c ) = 0$ for all $G \in \cG$. Consider $f=\chi_A$. As $\mu(A)>0$ and $\mu(A^c)>0$, this function is not equal to a constant almost everywhere. However, the proof of Lemma \ref{lemmanotnorm} gives $\|f\|_{\BMOp} = 0$. Thus, this function violates \eqref{normhomogeneity}, showing that $\|\cdot\|_{\BMOp}$ is not a norm modulo constants.
\end{proof}

\begin{lemma}\label{norm+sigma=FCP}
If $\|\cdot\|_{\BMOp}$ is a norm modulo constants on $\BMOp$ and $(\X,\cG)$ is $\sigma$-decomposable, then $(\X,\cG)$ has the FCP.
\end{lemma}

\begin{proof}
Consider a decomposition $\cE$ of $\X$ subordinate to $\cG$, which exists by Proposition \ref{eomega}. It follows from Lemma \ref{decompositionsfinitelyconn} that if $\|\cdot\|_{\BMO}$ is a norm on $\BMO$, then for every pair $( E , \widetilde{E} )$ from $\{ E_{\omega} \}_{\omega \in \Omega}$, there exists a finite subset $\{ E_{{\omega}_n} \}_{n=0}^N \subset \{ E_{\omega} \}_{\omega \in \Omega}$ such that $E_{\omega_0}=E$ and $E_{\omega_{N}}=\widetilde{E}$ and for every $1\leq n\leq N$ there exists a $G_{n} \in \cG$ such that $\mu (G_{n} \cap E_{{\omega}_n} ) > 0$ and $\mu (G_{n} \cap E_{{\omega}_{n-1}} ) > 0$. 

Since $\cE$ is subordinate to $\cG$, every $E_{\omega_n}$ is contained in some $G_{\omega_n}\in\cG$. This gives us a finite chain between any pair of sets in $\{ G_{\omega} \}_{\omega \in \Omega}$. Since $\{ E_{\omega} \}_{\omega \in \Omega}$ is a decomposition, then we get that for any set $G$, there exists an $E_{\omega}$ such that $\mu ( G \cap E_{\omega} )>0$, and so $\mu ( G \cap G_{\omega} ) \geq \mu ( G \cap E_{\omega} ) > 0$. Therefore by transitivity, it follows that any pair of sets in $\cG$ is connected by a finite chain, and so $\X$ has the FCP.
\end{proof}

The following two propositions, when combined, show that (i) implies (iii) in Theorem \ref{completeness}. 

\begin{proposition}\label{normFCPsigma1}
If $(\X,\cG)$ has the FCP and is $\sigma$-decomposable, then $\|\cdot\|_{\BMOp}$ is a norm modulo constants on $\BMOp$. 
\end{proposition}

\begin{proof}
By $\sigma$-decomposability, we may consider a decomposition $\cE$ subordinate to $\cG$ and a collection $\{ G_{\omega} \}_{\omega \in \Omega}\subset\cG$ satisfying $E_\omega\subset G_\omega$ for all $\omega\in\Omega$. Let $f\in\BMOp$ satisfy $\|f\|_{\BMOp} = 0$. Then, for any $G\in\cG$, $f$ is constant almost everywhere on $G$. Applying this to some $G_0 \in \{ G_{\omega} \}_{\omega \in \Omega}$, we have that $f=f_{G_0}$ almost everywhere on $G_0$. From Proposition \ref{FCPconstant}, since $(\X,\cG)$ has the FCP, $f = f_{G_0}$ a.e. on every $G_{\omega}$, $\omega \in \Omega$. Since every $E_{\omega}$ is contained in some $G_{\omega}$, and the $\{E_\omega\}$ cover $X$ except for a set of measure zero, it follows that $f=f_{G_0}$ a.e. on $X$. As this shows \eqref{normhomogeneity}, $\|\cdot\|_{\BMOp}$ is a norm modulo constants on $\BMOp$.
\end{proof}

\begin{proposition}\label{completeFCPsigma}
If $(\X,\cG)$ has the FCP and is $\sigma$-decomposable, then $\BMOp$ is complete with respect to $\|\cdot\|_{\BMOp}$.
\end{proposition}

\begin{proof}
Let $\{f_i\}\subset\BMOp$ be Cauchy with respect to $\|\cdot\|_{\BMOp}$. Then, for every $G \in \cG$,
$$
\fint_G\! |(f_i - f_j ) - ( f_i - f_j )_G|^p \, d\mu = \fint_G\! | ( f_i - ( f_i )_G ) - ( f_j - ( f_j )_G )|^p \, d\mu \rightarrow 0\,,
$$  
as $i,j\rightarrow\infty$. This implies that, for a fixed $G$, $\{ f_i - ( f_i )_G\}$ is Cauchy in $\Lp(G)$, and so converges to a function $f^G\in\Lp(G)$ by completeness. Moreover, this limit function has mean 0.

Now, if we have two generators $G,G'\in\cG$ such that $\mu ( G \cap G' ) > 0$, observe that
\begin{equation}\label{cauchytoconv}
\lim\limits_{i \rightarrow \infty} \left[ ( f_i )_{G} - ( f_i )_{G'} \right ]
= \lim\limits_{i \rightarrow \infty} \left[ ( f_i - ( f_i )_{G'} ) - ( f_i - ( f_i )_{G} ) \right] = f^{G'} - f^{G}
\end{equation}
in $\Lp(G\cap G')$. As $( f_i )_{G} - ( f_i )_{G'}$ is a sequence of constants that is Cauchy in $\Lp(G\cap G' )$, we have that $f^{G} - f^{G'}$ is equal to a constant almost everywhere on $G\cap G'$. We will denote this constant as $C (G, G' )$. Observe that there is an antisymmetry property to these constants: $C(G,G') = - C(G',G)$.

By $\sigma$-decomposability, we may consider a decomposition $\cE$ that is subordinate to $\cG$ and a collection $\{G_{\omega} \}_{\omega \in \Omega}\subset\cG$ satisfying $E_\omega\subset G_\omega$ for all $\omega\in\Omega$. Fix some $\omega_0 \in \Omega$. We use the transitive property of the FCP in order to associate a constant $C_G$ to each $G \in \cG$ that depends only on $G$ and the choice of $\omega_0$. For $G \in \cG$, the FCP gives a finite chain $\{ G_n \}_{n=0}^N$ connecting $G_{\omega_0}$ to $G$, where $G_{\omega_0} = G_0$ and $G = G_N$. Write 
$$
C_G = \sum\limits_{n=0}^N C(G_n , G_{n-1})\,.
$$

First, we show that $C_G$ is well defined, regardless of the chain chosen. Let $\{ G_k \}_{k=0}^K$ be another finite chain connecting $G$ and $G_{\omega_0}$, and write 
$$
C_G^{\ast} = \sum\limits_{k=0}^K C(G_k , G_{k-1})\,.
$$
Since $C_G$ and $C_G^{\ast}$ are limits of telescoping sums, we have that
$$
\sum\limits_{k=0}^K \lim\limits_{i \rightarrow \infty} [(f_i)_{G_k} - (f_i)_{G_{k-1}}] = \lim\limits_{i \rightarrow \infty} [(f_i)_G - (f_i)_{G_{\omega_0}}] = \sum\limits_{n=0}^N \lim\limits_{i \to \infty} [(f_i)_{G_n} - (f_i)_{G_{n-1}}] \,,
$$ 
and so $C_G = C_G^{\ast}$.

Next, we want to show that if $\mu(G \cap G') > 0$, then on $G \cap G'$, 
\begin{equation}\label{limitequality}
f^G + C_G = f^{G'} + C_{G'}
\end{equation}
We show this by using \eqref{cauchytoconv} and introducing a telescoping sum with respect to the finite chains $\{ G_n \}_{n=0}^N$ and $\{ G_n \}_{j=0}^J$, where $G_0 = G_{\omega_0}$ in both chains, $G_N = G$ and $G_J = G'$. On $G \cap G'$,
\[
\begin{split}
f^{G'} - f^G
&= \lim\limits_{i \to \infty} [(f_i)_G - (f_i)_{G'}]
= \sum\limits_{n=0}^N \lim\limits_{i \to \infty} [(f_i)_{G_n} - (f_i)_{G_{n-1}}] + \sum\limits_{j=0}^J \lim\limits_{i \to \infty} [(f_i)_{G_{j-1}} - (f_i)_{G_j}]\\
&= \sum\limits_{n=0}^N \lim\limits_{i \to \infty} [(f_i)_{G_n} - (f_i)_{G_{n-1}}] - \sum\limits_{j=0}^J \lim\limits_{i \to \infty} [(f_i)_{G_j} - (f_i)_{G_{j-1}}]
= C_G - C_{G'}.
\end{split}
\]
This shows \eqref{limitequality}.

Taking $G'=G_{\omega}$ for some $\omega\in\Omega$, we can at last build the limit function:
$$
f = \sum\limits_{\omega \in \Omega} (f^{G_{\omega}} + C_{G_{\omega}}) \chi_{E_{\omega}}\,.
$$
Note that since the sets $\{ E_{\omega} \}_{\omega \in \Omega}$ are disjoint, this uncountable sum is in fact only a finite sum of one element, for any $x \in X$, hence it is well defined.

Next, we show that for $G \in \cG$, for almost every $x \in G$,
\begin{equation}\label{limitfunction}
f(x)-f_G = f^G(x).
\end{equation}
First, since $\mu(G) < \infty$, by Proposition \ref{finitedecomposition}, there exists a countable $\{E_{\omega_k} \}_{k=1}^{\infty}$ such that $\mu (G) = \sum\limits_{k=1}^{\infty} \mu (G \cap E_{\omega_k})$. Therefore, using \eqref{limitequality}, for almost every $x \in G$, we have
\[
\begin{split}
f(x)-f_G
&= \sum\limits_{k=1}^{\infty} (f^{G_{\omega_k}}(x) + C_{G_{\omega_k}}) \chi_{E_{\omega_k}} - \frac{1}{\mu(G)} \sum\limits_{k=1}^{\infty}\left[ \int_{G \cap E_{\omega_k}}\! (f^{G_{\omega_k}}(x) + C_{G_{\omega_k}})\right] d\mu \\
&= \sum\limits_{k=1}^{\infty} (f^G(x) + C_G) \chi_{E_{\omega_k}} - \frac{1}{\mu(G)} \sum\limits_{k=1}^{\infty}\left[ \int_{G \cap E_{\omega_k}}\! (f^G(x) + C_G) \right] d\mu= f^G(x)\,.
\end{split}
\]

Finally, by \eqref{limitfunction} and by the definition of $f^G$ as the limit of $f_i - (f_i)_G$ in $\Lp(G)$, we get that for every $G \in \cG$
\[
\begin{split}
\int_G\! |(f_i-f) - (f_i-f)_G|^p \, d\mu
&= \int_G\! |(f_i - (f_i)_G) - (f - f_G)|^p \, d\mu \\
&= \int_G\! |(f_i - (f_i)_G) - f^G|^p \, d\mu \rightarrow 0 \quad\text{as}\quad i\rightarrow\infty\,,
\end{split}
\]
from where the result follows.
\end{proof}

To finish the proof of Theorem \ref{completeness}, it remains to show that (ii) implies (i). This is addressed in the following proposition.

\begin{proposition}\label{normFCPsigma2}
If $\|\cdot\|_{\BMOp}$ is a norm on $\BMOp$, then $(\X,\cG)$ has the FCP and is $\sigma$-decomposable.
\end{proposition}

\begin{proof}
By Lemma \ref{norm+sigma=FCP}, we only need to show $\sigma$-decomposability. Fix a decomposition $\om$ of $\X$. We can assume, without loss of generality, that $\mu ( X_{\omega} ) > 0$ for every $\omega \in \Omega$. For $\omega \in \Omega$,  let $\mathscr{A}_{\omega} = \{ G \in \cG : \mu  ( G \cap X_{\omega} ) > 0 \}$. By Lemma \ref{lemmanotnorm}, $\mathscr{A}_{\omega}\neq \emptyset$. 

Fix some $\omega \in \Omega$ and let $\widetilde{Q}$ denote a countable union of intersections of elements of $\mathscr{A}_{\omega}$ with $X_{\omega}$. Since $\mu ( X_{\omega} ) < \infty$, we know that  $K:=\sup \{ \mu ( \widetilde{Q} ):\widetilde{Q} \subset X_{\omega} \}$ exists. We aim to show that there exists a possibly non-unique maximal set $Q$; that is, a countable union of intersections of elements of $\mathscr{A}_{\omega}$ with $X_{\omega}$ and such that $\mu ( Q ) = K$.

We construct $Q$ as follows. For $i \in \N$, take $\widetilde{Q}_i$ such that $\mu (\widetilde{Q}_i ) > K - \frac{1}{i}$. Then, write 
$$
Q_n = \bigcup\limits_{i=1}^n \widetilde{Q}_i\quad \text{and}\quad Q = \bigcup\limits_{n=1}^{\infty} Q_n\,.
$$
As $Q$ is a countable union of intersections of elements of $\mathscr{A}_{\omega}$ with $X_{\omega}$, $\mu(Q)\leq K$. On the other hand, as the $\{Q_n\}$ form an increasing sequence of measurable sets satisfying $ K - \frac{1}{n}\leq \mu(Q_n)\leq \mu(X_\omega)<\infty$, we know that
$$
\mu(Q)= \lim_{n\rightarrow\infty}\mu(Q_n)\geq \lim_{n\rightarrow\infty}\left(K - \frac{1}{n}\right)=K\,.
$$
Therefore, $\mu(Q)=K$. 

Let $E = X_{\omega} \setminus Q$. If there exists a $G\in\cG$ such that $\mu ( G \cap E) >0$, then $\mu ( ( G \cap X_{\omega} ) \cup Q ) = \mu ((G \cap E) \cup Q) = \mu (G \cap E) + \mu (Q) > K$, which contradicts the maximality of $Q$, and so $\mu ( G \cap E ) = 0$ for every $G \in \cG$.

By Lemma \ref{lemmanotnorm}, if $\mu ( E ) > 0$ then $\|\cdot\|_{\BMOp}$ is not a norm modulo constants on $\BMOp$, which contradicts the hypothesis of the proposition. Thus, $\mu(Q)=\mu(X_{\omega})$. As this is true for any $\omega\in\Omega$, we have that $\cG$ admits a countable essential subcover of $X_{\omega}$, which means that $(\X,\cG)$ is $\sigma$-decomposable.
\end{proof}

Now we present an immediate corollary of Theorem \ref{completeness}.

\begin{corollary}
Let $p\in(1,\infty)$. Then $\BMO$ is a Banach space if and only if $\BMOp$ is a Banach space. 
\end{corollary}

\begin{proof}
This follows from the fact that $\sigma$-decomposability and the FCP do not depend on $p$, so long as $p \in [1,\infty)$.
\end{proof}

\subsection{Connected topological measure spaces}

This section deviates from the assumptions on the pair $(\X,\cG)$. We consider measure spaces that are compatible with a topology on the set $X$:

\begin{definition}\label{toplocint}
We say that $(\X,\tau)$ is a topological measure space if $\X$ is a measure space and $\tau$ is a topology such that $\cM$ contains the associated Borel $\sigma$-algebra.
\end{definition}

In this setting, it makes sense to require that $\cG$ is comprised of open sets. We denote by $\mathrm{BMO}_{\cG}^p(\X,\tau)$ the BMO space corresponding to the topological measure space $(\X,\tau)$ and the open cover $\cG$.

\begin{theorem}\label{topologytheorem}
Let $(\X,\tau)$ be a topological measure space and $\cG$ an open cover of $X$ that satisfy the following:
\begin{itemize}
\item[(i)] $X$ is connected,
\item[(ii)] $\mu(G_1 \cap G_2) > 0$ for every $G_1,G_2 \in \cG$ such that $G_1 \cap G_2 \neq \emptyset$, and
\item[(iii)] $(X,\tau)$ is Lindel\"{o}f.
\end{itemize}
Then $\mathrm{BMO}_{\cG}^p(\X,\tau)$ is a Banach space.
\end{theorem}

\begin{proof}
Recall that the Lindel\"{o}f property means that any open cover has a countable subcover. Therefore, $\cG$ has a countable subcover of $X$. From the fact that generators in $\cG$ have positive and finite measure, we get that $\X$ is $\sigma$-finite, hence decomposable, and that $(\X,\cG)$ is $\sigma$-decomposable. 

To show the FCP, fix a $G_0 \in \cG$ and let $\mathscr{A} = \{ G \in \cG: G \; \text{is connected by a finite chain to} \; G_0 \}$, $A = \bigcup\limits_{G \in \mathscr{A}} G$, $\mathscr{B} = \cG \setminus \mathscr{A}$, and $B = \bigcup\limits_{G \in \mathscr{B}} G$. Assume $\mathscr{B} \neq \emptyset$, then $A \cap B \neq \emptyset$, since $X$ is connected. This implies that there exists a set $G_a \in \mathscr{A}$ and a set $G_b \in \mathscr{B}$ such that $G_a \cap G_b \neq \emptyset$. So by (ii), $\mu(G_a \cap G_b) > 0$, which is a contradiction, because this would imply that $G_b$ is connected to $G_0$ by a finite chain, hence that $G_b \in \mathscr{A}$. So we get that $\mathscr{A} = \cG$, hence by the transitivity property of finite chains, we have that $\cG$ has the FCP.

Finally, by Theorem \ref{completeness}, $\mathrm{BMO}_{\cG}^p(\X,\tau)$ is a Banach space.
\end{proof}

Note that if $(\X,\tau)$ is such that $\mu(U)>0$ for every open $U \subset X$, then (ii) is implied, since $G_1 \cap G_2$ is open. Recall that a topology $\tau$ is metrizable if there exists a metric such that the topology $\tau_B$ generated by the open balls of the metric is equivalent to $\tau$, i.e. $U$ is open in $\tau$ if and only if it is open in $\tau_B$. We will use the notation $B(x,r)$ to denote an open ball centred at $x\in{X}$ of radius $r>0$.

\begin{definition}\label{defccc}
We say that a topological space has the countable chain condition (CCC) if any collection of disjoint open sets is at most countable.
\end{definition}

It is a well known fact that in a metric space, CCC, separability, and the Lindel\"{o}f properties are all equivalent, and that any subset also inherits these properties.

\begin{corollary}\label{metriccorollary}
Let $(\X,\tau)$ be a metrizable topological measure space and $\cG$ an open cover of $X$ that satisfy the following:
\begin{enumerate}
\item[(i)] $X$ is connected and
\item[(ii)] $\mu(B) > 0$ for every ball $B \subset X$ in $\tau_B$.
\end{enumerate}
Then $\mathrm{BMO}^p_{\cG}(\X,\tau)$ is a Banach space.
\end{corollary}

\begin{proof}
First, we claim that if $B_0$ is a non-separable ball, then $\mu(B_0) = \infty$. If $B_0$ is not separable, then it does not possess the CCC, so there exists an uncountable collection of disjoint open sets contained in $B_0$. By (ii), these open sets must have strictly positive measure, and so $\mu(B_0) = \infty$.

Second, we claim that for every $x \in X$, there exists $r>0$ such that $B(x,r)$ is separable. Assume the opposite, that there exists an $x \in X$ such that $B(x,\frac{1}{n})$ is non-separable for every $n \in \N$. Since $\cG$ is an open cover of $X$, there exists a $G \in \cG$ and an $N \in \N$ such that $B(x,\frac{1}{N}) \subset G$. Therefore, by the preceding paragraph and monotonicity, we get that $\mu(G) = \infty$, which contradicts the measure assumption on generators.

Next, we prove that $(X,\tau)$ is Lindel\"{o}f. Assume not. Then, it is not separable. By the preceding paragraph, we can fix a separable ball $B_0$ around which we will build a separable set $S$.

Let $S_0 = B_0$ and $Q_0$ be a fixed countable dense subset of $S_0$. For each $q \in Q_0$, let $r_q = \sup\{r>0: B(q,r) \, \text{is separable}\}$. If $r_q = \infty$ for some $q \in Q_0$, then $X= \bigcup\limits_{n=1}^{\infty} B(q,n)$ is separable. Hence, for every $q \in Q_0$, $r_q < \infty$. Then $B(q,r_q) = \bigcup\limits_{n=1}^{\infty} B(q,r_q-1/n)$ is separable. Now, let $S_1 = \bigcup\limits_{q \in Q_0} B(q,r_q)$. Since it is a countable union, $S_1$ is separable. Fix a countable subset $Q_1$ dense in $S_1$ such that $Q_1 \cap S_0 = Q_0$. Then, let $S_2 = \bigcup\limits_{q \in Q_1} B(q,r_q)$. Repeat this process indefinitely to obtain a separable set $S = \bigcup\limits_{n=0}^{\infty} S_n$. Since $X$ is non separable, then $T=X \setminus S \neq \emptyset$.

Assuming now that dist$(x,S) = \delta_x > 0$ for every $x \in T$, it follows that $T = \bigcup\limits_{x \in T} B(x,\delta_x)$, which is open. Since $S$ is also open, this is a contradiction, because $X$ is connected. Therefore, there exists an $x \in T$ such that dist$(x,S) = 0$, which then implies that $B(x,\frac{1}{2n}) \cap S \neq \emptyset$ for every $n \in \N$. Then, by density, there exists a $q_0 \in B(x,\frac{1}{2n}) \cap \left( \bigcup\limits_{n=0}^{\infty} Q_n \right)$. Therefore, by the construction of $S$, $B(q_0,\frac{1}{2n})$ is non separable, since it contains $x \in T$. By the triangle inequality, $B(q_0,\frac{1}{2n}) \subset B(x,\frac{1}{n})$, which implies that $B(x,\frac{1}{n})$ is non-separable. This contradicts the existence of a separable ball, and so $(X,\tau)$ is Lindel\"{o}f.

Therefore, $(\X,\tau,\cG)$ satisfies all the conditions of Theorem \ref{topologytheorem}.
\end{proof}


\section{Muckenhoupt weight Classes and JNP}

Consider a pair $(\X,\cG)$ comprised of a semifinite measure space $\X$ and a fixed collection of generators $\cG$. As mentioned in the introduction, Muckenhoupt weights with respect to general bases have been well studied. We review some basic facts about such weights before investigating their relationship to $\BMO$. 

\subsection{Muckenhoupt weight classes}

\begin{definition}\label{defap}
We say that a nonnegative measurable function $w$ is an $A_p(\X,\cG)$ weight, $p \in (1,\infty)$, if 
$$
[w]_{A_p} = \sup\limits_{G \in \cG} \left(\fint_G\! w \, d\mu \right) \left(\fint_G \!w^{\frac{-1}{p-1}} \, d\mu \right)^{p-1} < \infty\,.
$$
\end{definition}

Note that if $w\in A_p(\X,\cG)$, then both $w\in \Loneloc(\X,\cG)$ and $w^{\frac{-1}{p-1}}\in \Loneloc(\X,\cG)$. This implies that $0<w<\infty$ almost everywhere. Also, note that the $A_p$ constant $[\cdot]_{A_p}$ is not a norm: by H\"{o}lder's inequality, it is always greater than or equal to 1.

For $f\in\Loneloc(\X,\cG)$, we can consider the weight $w=e^f$. By Jensen's inequality, it satisfies
$$
[w]_{A_p} \leq \sup\limits_{G \in \cG} \left(\fint_G\! e^{f-f_G} \, d\mu \right) \sup\limits_{G \in \cG} \left(\fint_G\! e^{\frac{f_G-f}{p-1}} \, d\mu \right)^{p-1} \leq [w]_{A_p}^2\,.
$$
Applying this to $p=2$ implies the following inequality:
\begin{equation}\label{A2}
([w]_{A_2})^{1/2} \leq \sup\limits_{G \in \cG} \fint_G\! e^{|f-f_G|} \, d\mu \leq 2[w]_{A_2}\,.
\end{equation}
Therefore, the quantity
$$
[f]_{\ast}:=\sup\limits_{G \in \cG} \fint_G \!e^{|f-f_G|} \, d\mu
$$
is finite if and only if $w=e^f\in A_2(\X,\cG)$. This characterization of $\log A_2(\X,\cG)$ is useful in the study of the interaction between Muckenhoupt weights and BMO. An easy consequence is the following. 

\begin{proposition}\label{loga2bmo}
$$\Linfty(\X) \subset \log A_2(\X,\cG) \subset \BMO$$
\end{proposition}

\begin{proof}
Both inclusions follow from \eqref{A2}. If $f \in \Linfty(\X)$, then Jensen's inequality implies that
$$
\fint_G\! e^{|f-f_G|} \, d\mu \leq \left( \fint_G \!e^{|f|} \, d\mu \right)^2 \leq e^{2\|f\|_{\infty}}\,
$$
for any $G\in \cG$. Hence, $[f]_{*}\leq e^{2\|f\|_{\infty}}$ showing that $f\in \log A_2(\X,\cG)$.

If $w = e^f \in A_2(\X,\cG)$, then $f\in\Loneloc(\X,\cG)$ and
$$
\fint_G\! |f-f_G| \, d\mu \leq \fint_G\! e^{|f-f_G|} \, d\mu
$$
for any $G \in \cG$. Hence, $\|f\|_{\BMO}\leq [f]_{*}$, showing that $f\in\BMO$. 
\end{proof}

There are examples in the literature to show that these inclusions can be proper, see \cite[Section II.3]{spanish}.

We introduce the following notation when $f\in\BMO$ with $\|f\|_{\BMO}\neq{0}$ and $\alpha>0$:
$$
[f]_{\alpha}=\sup\limits_{G \in \cG} \fint_G\! \exp\left({\frac{{\alpha}}{\|f\|_{\BMO}}|f-f_G|}\right) d\mu\,.
$$
With this notation, $[f]_{\ast}$ corresponds to $[f]_{\alpha}$ when $\alpha=\|f\|_{\BMO}$.

\subsection{$A_2$ and the John-Nirenberg property}

We show that for every $f\in \BMO$, having $cf\in\log A_2$ for sufficiently small $c$ is equivalent to $f$ satisfying a John-Nirenberg inequality.

We begin by defining the John-Nirenberg property for the space $\BMO$. 

\begin{definition}\label{MRD}
For any $E \in \cM^{\ast}$ and $f$ measurable, let 
$$
\mu_f(t,E) = \frac{\mu(\{x \in E: |f(x)-f_E|>t \})}{\mu(E)}
$$ 
be called the mean relative distribution function of $f$ with respect to $E$.
\end{definition}

\begin{definition}\label{defWNJP}
We say that $\BMO$ has the weak John-Nirenberg property (WJNP) if there exists $c>0$ such that for every $f \in \BMO$ there is a $C_f>0$ for which
$$
\mu_f(t,G) \leq C_f \exp \left( \frac{-c}{\|f\|_{\BMO}}t \right)\,
$$
holds for every $G \in \cG$ and $t>0$.
\end{definition}

\begin{definition}\label{defJNP}
We say that $\BMO$ has the John-Nirenberg property (JNP) if there exist $c_1, c_2 >0$ such that for every $f \in \BMO$,
$$
\mu_f(t,G) \leq c_1 \exp \left( \frac{-c_2}{\|f\|_{\BMO}}t \right)\,
$$
for every $G \in \cG$ and $t>0$.
\end{definition}

Essentially, the WJNP states that the mean relative distribution function of any $f \in \BMO$ relative to any $G \in \cG$ is exponentially decreasing. The JNP states that this decay is independent of $f$. This decay allows us to relate $\BMOp$ for different values of $p$. Note that there is a trivial relationship available for any pair $(\X,\cG)$: by Jensen's inequality, 
\begin{equation}\label{trivialdirection}
\BMOp \subset \BMO\quad\text{with}\quad\|f\|_{\BMO} \leq \|f\|_{\BMOp}\,.
\end{equation}

\begin{proposition}\label{isomorphism}
$$\mathrm{WJNP} \implies \BMO = \BMOp, \; \text{for every} \; p \in [1,\infty)$$
$$\mathrm{JNP} \implies \BMO \cong \BMOp, \; \text{for every} \; p \in [1,\infty)$$
\end{proposition}

Here $\BMO = \BMOp$ means equality in terms of sets, while $\BMO \cong \BMOp$ means that the spaces are also isomorphic.

\begin{proof}
Assume that $\BMO$ has the WJNP. Then,
$$
\left( \fint_G\! |f-f_G|^p \, d\mu \right)^{1/p}= \left( p \int_0^{\infty} t^{p-1} \mu_f(t,G) \, dt \right)^{1/p}\leq \frac{\left( C_f p \Gamma (p) \right)^{1/p}}{c} \|f\|_{\BMO}\,.
$$
Combining this with \eqref{trivialdirection}, $\BMO = \BMOp$ with $\|f\|_{\BMO}\leq \|f\|_{\BMOp} \leq K(f,p) \|f\|_{\BMO}$. The constant $K(f,p)$, however, may depend on $f$, so this does not tell us whether the spaces are isomorphic.

On the other hand, if $\BMO$ has the JNP, then we may replace $C_f$ by $c_1$ and $c$ by $c_2$, so that $\|f\|_{\BMO}\leq \|f\|_{\BMOp} \leq K(p) \|f\|_{\BMO}$ holds with constant $K(p) = \frac{\left( c_1 p \Gamma(p) \right)^{1/p}}{c_2}$, independent of $f$. Therefore, $\BMO \cong \BMOp$.
\end{proof}

\begin{corollary}\label{JNPlploc}
If $\BMO$ has the WJNP, then $\BMO \subset \Lploc(\X,\cG)$ for any $p \in (1, \infty)$.
\end{corollary}

Now, we turn to the property of $\BMO$ concerning $A_2$ weights that we will connect to the JNP in Theorem \ref{jnpshs}. 

\begin{definition}\label{defWHS}
We say that $\BMO$ has the weak Helson-Szeg\H{o} property (WHSP) if there exists $c_0 >0$ such that for every $f \in \BMO$, 
$$
\frac{{\alpha}}{\|f\|_{\BMO}}f \in \log A_2(\X,\cG)\,
$$
for every $\alpha \in (0,c_0)$.
\end{definition}

\begin{definition}\label{defSHS}
We say that $\BMO$ has the strong Helson-Szeg\H{o} property (SHSP) if
\begin{enumerate}
\item it has the WHSP with some constant $c_0>0$, and
\item for every $\alpha \in (0,c_0)$, there exists a constant $C_{\alpha} >0$ such that $[f]_{\alpha} \leq C_{\alpha}$ for every $f \in \BMO$.
\end{enumerate}
\end{definition}

Finally, we present the main theorem of the section, relating the John-Nirenberg and Helson-Szeg\H{o} properties. 

\begin{theorem}\label{jnpshs}
$$\mathrm{WHSP} \; \iff \; \mathrm{WJNP}$$
$$\mathrm{SHSP} \; \iff \; \mathrm{JNP}$$
\end{theorem}

\begin{proof}
By Chebyshev's inequality,
$$
\mu_f(t,G)\leq \exp \left(\frac{-{{\alpha}}}{\|f\|_{\BMO}}t \right)\fint_G\! \exp\left({\frac{{{\alpha}}}{\|f\|_{\BMO}}|f(x)-f_G|}\right) d\mu\,.
$$
From this it follows that if $\BMO$ has the WHSP, then fixing $\alpha \in (0,c_0)$ as in Definition \ref{defWHS} gives 
$$
\mu_f(t,G) \leq [f]_{\alpha} \exp \left( \frac{- \alpha}{\|f\|_{\BMO}}t \right)\,,
$$
which implies the WJNP with $C_f = [f]_{\alpha}$ and $c = \alpha$, which does not depend on $f$.

If $\BMO$ has the SHSP, then fixing $\alpha<c$ gives
$$
\mu_f(t,G) \leq C_{\alpha} \exp \left( \frac{-\alpha}{\|f\|_{\BMO}}t \right)\,,
$$
which implies the JNP with $c_1 = C_{\alpha}$ and $c_2 = {\alpha}$, since $\alpha$ does not depend on $f$.

The proof of the converse directions is identical to those of the corresponding statements for classical BMO and $A_2$ (see \cite[Corollary 3.10 in Section II.3]{spanish}). If $\BMO$ has the WJNP, then
$$
\fint_G\! \exp\left({\frac{\alpha}{\|f\|_{\BMO}}|f-f_G|}\right) d\mu \leq 1+ \frac{C_f}{(c/ \alpha) -1}\,
$$
for any $\alpha \in (0,c)$. Therefore, $\BMO$ has the WHSP with $c = c_0$. Similarly, if $\BMO$ has the JNP, then
$$
\fint_G \! \exp\left({\frac{\alpha}{\|f\|_{\BMO}}|f-f_G|}\right) d\mu \leq 1+ \frac{c_1}{(c_2/ \alpha) -1}$$
for any $\alpha \in (0,c_2)$. Therefore, $\BMO$ has the SHSP with $c = c_2$ and $C_{\alpha} = 1 + \frac{c_1}{(c_2 / \alpha) -1}$.
\end{proof}


\section{Denjoy Families and JNP}
In this section, we describe a large class of collections $\cG$ for which the associated BMO space possesses the John-Nirenberg property (JNP). We call these families Denjoy since their most important property (called the Denjoy doubling property) is taken from an article by Arnaud Denjoy in \cite{denjoy}. His idea was to generalize Vitali's covering theorem without using any notion of distance, only measure-theoretic tools. 

In a subsequent section, we will see how the results here can be used to ascertain the JNP for a number of BMO spaces, notably in $\R^n$ with Lebesgue measure and the collection $\cG$ of all rectangles of arbitrary orientation (see Example \ref{rectangles}), and in spaces of homogeneous type with the collection $\cG$ of all open balls (see Example \ref{homogeneous}). 

\subsection{The setting}

For $x\in{X}$, write $\cG(x)=\{G\in\cG:G\ni x \}$. We say that a sequence $\{G_n\}\subset\cG(x)$ converges to $x$, written $G_n \to x$, if $\mu(G_n) \to 0$. Similarly, we write 
$$
\limsup\limits_{G \to x} \fint_G\! |f| \, d\mu = \lim\limits_{\eps \to 0} \sup\limits_{\mu(G) < \eps} \fint_G\! |f| \, d\mu\,,
$$
where the supremum is taken over all $G \in \cG(x)$ such that $\mu(G) < \eps$.

The outer measure $\mu^{\ast}$ associated to the measure $\mu$ is defined for $A\subset X$ by the formula
$$
\mu^{\ast}(A) = \inf \{\mu(E): A \subset E \in \cM\}\,.
$$ 
Note that $\mu^{\ast}(E) = \mu(E)$ for any $E \in \cM$ and the $\sigma$-algebra of $\mu^*$-measurable sets contains $\cM$.

We will also define a set map $\Omega_a$, $a>1$, for $A\subset X$ by 
$$
\Omega_a(A) = \bigcup \{G \in \cG: G \cap A \neq \emptyset \; \text{and} \; \mu(G) \leq a\mu^{\ast}(A)\}\,.
$$
The notation $\Omega_a^{(j)}(A)$ is used to denote the $j^{th}$ iteration of the map $\Omega_a$. One could interpret $\Omega_a(A)$ as some type of ``fattening" of the set $A$, with respect to $\cG$. 

\begin{definition}\label{double}
Fix $a,b$ such that $1< a \leq b$. We will say that a generator $G' \in \cG$ is an $(a,b)$-Denjoy double of $G \in \cG$ if $\Omega_a(G) \subset G'$ and $\mu(G') \leq b\mu(G)$. Whenever $a$ and $b$ are fixed, we will simply use the term double.
\end{definition}

Note that a generator may have more than one double, but in all our proofs, the choice of double is irrelevant. Therefore when writing $G'$, we will mean some fixed double of $G$.

As an example, if we let $\cG$ be the collection of cubes in $\R^n$ with Lebesgue measure and let $a=2^n$, then we have $\Omega_a(Q) = 5Q$ for every cube $Q$ and so $\mu(\Omega_a(Q)) = 5^n\mu(Q)$. Then, if we let $b=6^n$, $cQ$ is a double of $Q$ for any $c \in [5,6]$.

On the other hand, if we replace cubes by rectangles, then, for any choice of $a>1$, we get that $\Omega_a(R)=\R^n$ for any rectangle $R$ and so $\mu(\Omega_a(R)) = \infty$. Hence, no rectangle has a double. As we will see in Example \ref{rectangles}, although rectangles are not a Denjoy family, we can still use the properties of Denjoy families locally to prove the JNP for $\BMO$, where $\cG$ is the collection of all rectangles.

\begin{definition}\label{defdenjoy}
For a pair $(\X,\cG)$, we will call $\cG$ a Denjoy family when it satisfies the following conditions:
\vskip 1\baselineskip
\begin{itemize}
\item[] {\bf Shrinking Property}\\ For almost every $x \in X$, there exists $G_n \to x$.\\
\item[] {\bf Denjoy Doubling Property}\\ There exist uniform constants $a$ and $b$, with $1 < a \leq b$, such that every $G \in \cG$ has an $(a,b)$-double.\\
\item[] {\bf Growth Property}\\ For any $G_0,G \in \cG$ such that $G \cap G_0 \neq \emptyset$, there exists $j \in \N$ such that 
$\mu^{\ast}\left(\Omega_a^{(j)}(G)\right) \geq \frac{a}{b} \mu(G_0).$\\
\item[] {\bf Weak Differentiation Property}\\ For every $f \in \Loneloc(\X,\cG)$, for almost every $x \in X$, $|f(x)| \leq \limsup\limits_{G \to x} {\displaystyle \fint_G} \!|f| \, d\mu.$
\end{itemize}
\end{definition}

The doubling property is, in a sense, the most restrictive condition, yet the most useful one. As mentioned previously, this property is an analogue to the classical doubling property for spaces of homogeneous type.

The growth property eliminates ``bad" families in which some set $G$ intersects only sets of relatively small or big measures, with respect to the measure of $G_0$. This will allow us to take maximal sets, in the sense defined in Lemma \ref{maximal} below. Notice that if $\mu(G) \geq \frac{a}{b}\mu(G_0)$, then $\mu^{\ast}\left(\Omega_a(G)\right) \geq \mu(G_0)$, since $G \subset \Omega_a(G)$. Therefore, the property is intended to ensure that we can fatten repetitively any set $G$, no matter how small relative to $G_0$.

\subsection{Proof of Theorem \ref{Denjoytheorem}}

The proof of this theorem follows the same strategy as the classical proof of the John-Nirenberg inequality, as presented in \cite[Theorem 3.8 of Chapter II]{spanish}. What needs to be shown, in order for the proof to carry out {\it mutatis mutandis}, are a decomposition of Calder\'{o}n-Zygmund type (Lemma \ref{calderonzygmund}) and a result that follows from it (Lemma \ref{almostJN}).

\begin{lemma}\label{maximal}
Let $\cG$ be a Denjoy family. If $\widetilde{G} \in \cG$ and $\mathscr{F} \subset \{G \in \cG: G \cap \widetilde{G} \neq \emptyset \; \& \; \mu(G) < \frac{a}{b} \mu(\widetilde{G})\}$ is nonempty, then there exists a (possibly non-unique) set $G \in \mathscr{F}$ with a double $G' \notin \mathscr{F}$.
\end{lemma}

\begin{proof}
Fix an arbitrary $G_0 \in \mathscr{F}$, and denote by $G_1$ a fixed double of $G_0$. If $G_1 \notin \mathscr{F}$, then the proof is done. Otherwise, we have
$$
\mu^{\ast}\left(\Omega^{(1)}_a(G_0)\right) \leq \mu(G_1) < \frac{a}{b} \mu(\widetilde{G})\,.
$$

Note that $\Omega_a$ is monotone with respect to inclusion, so $\Omega_a^{(2)}(G_0) = \Omega \left( \Omega^{(1)}_a(G_0) \right) \subset \Omega_a(G_1)$. Now, consider a sequence of doubles $G_j = (G_{j-1})'$. We cannot have all of the $\{G_j\}$ contained in $\mathscr{F}$, since that would mean 
$$
\mu^{\ast}\left(\Omega^{(j)}_a(G_0)\right) \leq \mu(G_j) < \frac{a}{b} \mu(\widetilde{G})\,
$$
for all $j\in\N$, in violation of the growth property. Hence, at least one of the $G_j$ satisfies the conclusion of the lemma.
\end{proof}

The main idea used here to show a Calder\'{o}n-Zygmund-type decomposition is taken from the proof of Denjoy's covering theorem in \cite{denjoy}. Like Denjoy, as we will see below, we need $a>1$ in order for us to take a relatively big set in a collection, such that all other sets in the collection that intersect it will be contained in the double.

\begin{lemma}[Calder\'{o}n-Zygmund-Type Decomposition]\label{calderonzygmund}
If $\cG$ is a Denjoy family, then there exist $k,K>0$ such that for every $G_0 \in \cG$, $f \in \BMO$ with $\|f\|_{\BMO}=1$, and $\alpha>1$, there is a countable pairwise disjoint collection $\{G_j\}_{j \in \mathcal{J}} \subset \cG$ with doubles $G_j'$ satisfying
\begin{enumerate}
\item[(i)] for almost every $x \in G_0 \setminus \left( \bigcup\limits_{j \in \mathcal{J}} G_j' \right)$, $|f(x)-f_{G_0}| \leq k\alpha$;
\item[(ii)] for every $j \in \mathcal{J}$, $|f_{G_j'} - f_{G_0}| \leq k\alpha$; and,
\item[(iii)] $\sum\limits_{j \in \mathcal{J}} \mu(G_j') \leq \left( \displaystyle{\frac{K}{\alpha}} \right) \mu(G_0)$.
\end{enumerate}
\end{lemma}

\begin{proof}
Fix an arbitrary $G_0 \in \cG$ and $f \in \BMO$ such that $\|f\|_{\BMO} =1$. We define the following:
\begin{itemize}
\item $h(x)=|f(x)-f_{G_0}|$,
\item $\mathscr{A}_0 = \{G \in \cG: G \cap G_0 \neq \emptyset \; \text{and} \; \mu(G)< \frac{a}{b}\mu(G_0) \}$,
\item $A = \bigcup \{G:G \in \mathscr{A}_0 \}$,
\item  $M_{G_0,f}(x) =
  \begin{cases}
    \sup \left\{\displaystyle{\fint_G\! h \, d\mu} : x \in G \in \mathscr{A}_0 \right\}, & \text{if $x \in A$}\\
     0, & \text{otherwise} \\
  \end{cases}$\\

\item $D_{\alpha} = \{x \in X: M_{G_0,f}(x)>\alpha \}$.
\end{itemize}

Note that by the shrinking property, almost every $x \in G_0$ is in $A$. Also, for any $G \in \mathscr{A}_0$, any double $G'$ is in $\Omega_a(G_0)$, since $\mu(G') \leq b\mu(G)<a\mu(G_0)$. Therefore, we can assume without loss of generality that 
$$
G_0 \subset A \subset \bigcup\limits_{G \in \mathscr{A}_0} G' \subset G_0'\,.
$$
Then, by the weak differentiation property, for almost every $x \in G_0$,
\begin{equation}\label{maxfunction}
h(x) \leq \limsup\limits_{G \to x} \fint_G\! h \, d\mu \leq  M_{G_0,f}(x)\,.
\end{equation}

For $\alpha>1$, let 
$$
\mathscr{A}_1 = \left\{G \in \mathscr{A}_0: \alpha < \fint_G \! h \, d\mu \right\} \quad \text{and} \quad d_1 = \sup\limits_{G \in \mathscr{A}_1} \{\mu(G)\}\,.
$$
We know that $d_1$ is finite, since the measure of the sets in the collection $\mathscr{A}_0$ is bounded by $\frac{a}{b}\mu(G_0)$. Note that $D_{\alpha} = \bigcup\limits_{\mathscr{A}_1} G$ and for every $G \in \mathscr{A}_1$, $\mu(G) < \alpha^{-1}\int_G h \, d\mu$.

Since $a>1$, we can pick $G_1 \in \mathscr{A}_1$ such that $d_1 < a\mu(G_1)$. This implies that if $G \in \mathscr{A}_1$ and $G\cap G_1 \neq \emptyset$, then $G \subset G_1'$. Furthermore, since the collection $\mathscr{F}=\mathscr{A}_1 \cap \{G \in \cG: \mu(G)>d_1/a\}$ satisfies the hypotheses of Lemma \ref{maximal}, we can pick this $G_1$ such that $G_1' \notin \mathscr{A}_1$.

Next, if we let $k = \frac{b^2}{a} + b$, we will show that
\begin{equation}\label{control}
\left[ G \in \mathscr{A}_1 \; \& \; G' \notin \mathscr{A}_1 \right] \implies \fint_{G'}\! h \, d\mu \leq k\alpha\,.
\end{equation}
Take $G$ that satisfies the left-hand side of \eqref{control}. Then either $\fint_{G'}\! h \, d\mu \leq \alpha$ or $G' \notin \mathscr{A}_0$, which means $\mu(G') \geq \frac{a}{b}\mu(G_0)$. In the latter case, $\mu(G_0')/\mu(G') \leq \frac{b^2}{a}$ and since $\|f\|_{\BMO} = 1 < \alpha$, we get
\begin{equation}\label{notA1}
\fint_{G'}\! h \, d\mu\leq \fint_{G'}\! |f-f_{G_0'}| \, d\mu + |f_{G_0} - f_{G_0'}| \leq \frac{\mu(G_0')}{\mu(G')} \fint_{G_0'}\! |f-f_{G_0'}| \, d\mu + \frac{\mu(G_0')}{\mu(G_0)} \fint_{G_0'}\! |f-f_{G_0'}| \, d\mu \leq \frac{b^2}{a} + b \leq k \alpha.
\end{equation}
This proves \eqref{control}.

Now we will build a sequence $\{G_n\}$ such that for every $n$, either $G_n' \notin \mathscr{A}_1$ or $G_n' \subset G_i'$ for some $1\leq i \leq n-1$. So we define, for $n \geq 2$, 
$$
\mathscr{A}_n = \left\{G \in \mathscr{A}_{n-1}: G \cap \left( \bigcup\limits_{i=1}^{n-1} G_i \right) = \emptyset \right\} \quad \text{and} \quad d_n = \sup\limits_{G \in \mathscr{A}_n} \{\mu(G)\}\,.
$$
As above, we can pick a $G_n \in \mathscr{A}_n$ such that $d_n < a\mu(G_n)$ and $G_n' \notin \mathscr{A}_n$. This implies that if $G \in \mathscr{A}_n \setminus \mathscr{A}_{n+1}$ (i.e. $G \in \mathscr{A}_n$ and $G\cap G_n \neq \emptyset$), then $G \subset G_n'$. Also, note that when constructed this way, the collection $\{G_n\}$ is pairwise disjoint. Furthermore, $\mathscr{A}_n \subset \mathscr{A}_{n-1}$ and if $G \in \mathscr{A}_1 \setminus \mathscr{A}_n$, then there exists $i \in \N$ such that $1 \leq i < n$ and $G \in \mathscr{A}_i \setminus \mathscr{A}_{i+1}$, hence $G \subset G_i'$.

We pick the subcollection of $\{G_n\}$, that we will label by $\{G_j\}_{j \in \mathcal{J}}$, satisfying $G_j' \notin \mathscr{A}_1$. By \eqref{control}, this will imply that $\fint_{G_j}\!h \, d\mu \leq k\alpha$ for every $j \in \mathcal{J}$. Moreover, we will show that $\bigcup\limits_n G_n' \subset \bigcup\limits_{j \in \mathcal{J}} G_j'$. These will be the sets $\{G_j\}$ in the statement of the lemma.

Since $G_n' \notin \mathscr{A}_n$, then either $G_n' \notin \mathscr{A}_1$, $\fint_{G_n'}\! h \, d\mu \leq k\alpha$, or otherwise $G_n' \in \mathscr{A}_1 \setminus \mathscr{A}_n$, which implies that there exists some $1\leq i \leq n-1$ such that $G_n' \subset G_i'$. Since we know that at least $G_1' \notin \mathscr{A}_1$, then for every $n$, either $G_n' \notin \mathscr{A}_1$ or there exists some $1 \leq i \leq n-1$ such that $G_n' \subset G_i'$ and $G_i' \notin \mathscr{A}_1$.

In other words, we have shown that for every $n$, $G_n' \subset G_j'$ for some $j \in \mathcal{J}$. If $\mathscr{A}_{n+1} = \emptyset$ for some $n \in \N$, then this implies that for every $G \in \mathscr{A}_1$, there exists $1\leq i \leq n$ such that $G \in \mathscr{A}_i \setminus \mathscr{A}_{i+1}$, implying that $G \in G_i'$. Therefore $D_{\alpha} = \bigcup\limits_{\mathscr{A}_1} G \subset \bigcup\limits_{i=1}^n G_i' \subset \bigcup\limits_{j} G_{j}'$.

Now assume that the process goes on indefinitely, i.e. $\mathscr{A}_n \neq \emptyset$ for all $n \in \N$. Since the $\{G_n\}$ are pairwise disjoint and contained in $G_0'$, then 
$$
\sum\limits_{n=1}^{\infty} \mu(G_n') \leq b \sum\limits_{n=1}^{\infty} \mu(G_n) < \infty\,,
$$
which implies that $\mu(G_n') \to 0$. Fix an arbitrary $G \in \mathscr{A}_1$. Then there is some $n \in \N$ such that $\mu(G) \geq a\mu(G_{n+1})$. Because we chose $G_{n+1}$ such that $d_{n+1} < a\mu(G_{n+1})$ and the $\{\mathscr{A}_i\}$ are nested, then $G \notin \mathscr{A}_i$ for every $i \geq n+1$. So there exists an $1\leq \ell \leq n$ such that $G \in \mathscr{A}_{\ell} \setminus \mathscr{A}_{\ell+1}$, implying that $G \subset G_{\ell}'$.

Hence, we obtain that $D_{\alpha} = \bigcup\limits_{\mathscr{A}_1} G \subset \bigcup\limits_{n=1}^{\infty} G_n' \subset \bigcup\limits_{j} G_{j}'$. Then, for almost every $x \in G_0 \setminus \left( \bigcup\limits_{j} G_j' \right)$, $$h(x) \leq M_{G_0,f}(x) \leq \alpha < k\alpha\,,$$ which proves (i).

As was already pointed out, by \eqref{control} 
$$
|f_{G_j'}-f_{G_0}| \leq \fint_{G_j'}\! |f-f_{G_0}| \, d\mu \leq k\alpha\,
$$
for every $j \in \mathcal{J}$, which proves (ii).

Finally, if we let $K=b^2(b+1)$, since the $\{G_j\} \subset \mathscr{A}_1$ are disjoint, using similar techniques to \eqref{notA1}, we get that
\begin{equation}\label{iii}
\sum\limits_{j=1}^{\infty} \mu(G_j') \leq b \sum\limits_{j=1}^{\infty} \mu(G_{j})< \frac{b}{\alpha} \sum\limits_{j=1}^{\infty} \int_{G_{j}}\! h \, d\mu \leq \frac{b}{\alpha} \int_{G_0'}\! |f-f_{G_0}| \, d\mu \leq \frac{K}{\alpha} \mu(G_0)\,,
\end{equation}
which proves (iii).
\end{proof}

For the following proof, we will define, for any $t>0$ and a fixed $f \in \BMO$ and a $G_0 \in \cG$, the sets 
$$
E_t = \{x \in G_0: |f(x)-f_{G_0}| > t\}\,.
$$

\begin{lemma}\label{almostJN}
If $\cG$ is a Denjoy family, then for every $G_0 \in \cG$, $f \in \BMO$ with $\|f\|_{\BMO}=1$, $\alpha>1$, and $N \in \N$,
\begin{equation}\label{keyinequality}
\mu(E_{Nk\alpha}) \leq \left(\frac{K}{\alpha} \right)^N \mu(G_0)\,,
\end{equation}
where $k=\frac{b^2}{a}+b$ and $K=b^2(b+1)$.
\end{lemma}

\begin{proof}
As in the proof of Lemma \ref{calderonzygmund}, set $h = |f-f_{G_0}|$. In order to prove the present lemma, we will apply Lemma \ref{calderonzygmund} repeatedly and make use of the fact that for any $t>0$, $\mu(E_t) \leq \mu(D_t)$ by \eqref{maxfunction}.

We will show by induction that for every $N \in \N$, we have a countable collection of sets $\{ G^N_j \} \subset \cG$ such that
\begin{enumerate}
\item[(i)] for almost every $x \in G_0 \setminus\bigcup\limits_j G^N_j$, $|f(x)-f_{G_0}| \leq Nk\alpha$;
\item[(ii)] for every $j$, $|f_{G^N_j} - f_{G_0}| \leq Nk\alpha$; and,
\item[(iii)] $\sum\limits_{j} \mu(G^N_j) \leq \left( \displaystyle{\frac{K}{\alpha}} \right)^N \mu(G_0)$.
\end{enumerate}
We will see that (ii) allows for lemma \ref{calderonzygmund} to be used iteratively. Then, since (i) implies that $E_{Nk\alpha} \subset \bigcup\limits_{j} G^N_j$, we will have that (iii) implies \eqref{keyinequality}.

We know it holds for $N=1$ by Lemma \ref{calderonzygmund}, where we have relabelled the sets $G'_j$ by $G^1_j$. Assume that (i), (ii), and (iii) hold for $N$. Fix $j$ and repeat the decomposition process of Lemma \ref{calderonzygmund} on each $G^N_j$ with the function $h_j(x) = |f(x)-f_{G^N_j}|$. Then for every $j$, we obtain a collection $\{G^N_{j,i} \} \subset \cG$ satisfying the conclusions of Lemma \ref{calderonzygmund}. This implies that for every $j$, for almost every $x \in G^N_j \setminus \left( \bigcup\limits_{i} G^N_{j,i} \right)$, 
$$
|f(x)-f_{G_0}| \leq |f(x)-f_{G^N_j}| + |f_{G^N_j}-f_{G_0}| \leq (N+1)k\alpha\,;
$$
and
$$
|f_{G^N_{j,i}}-f_{G_0}| \leq |f_{G^N_{j,i}}-f_{G^N_j}| + |f_{G^N_j}-f_{G_0}| \leq (N+1)k\alpha\,.
$$
Moreover, we have that
$$
\sum\limits_{j} \sum\limits_{i} \mu(G^N_{j,i}) \leq \frac{K}{\alpha} \sum\limits_{j} \mu(G^N_j) \leq \left( \frac{K}{\alpha} \right)^{(N+1)} \mu(G_0)\,.
$$

If we relabel $\{G^N_{j,i} \}_{i,j}$ as $\{G^{N+1}_j\}_j$, we see that it satisfies (i), (ii), and (iii) above for $N+1$. Therefore, by induction, they hold for every $N \in \N$.
\end{proof}

As pointed out earlier, this lemma and the preceding Calder\'{o}n-Zygmund decomposition are enough to be able to prove Theorem \ref{Denjoytheorem}. Interestingly, Example \ref{noJNPexample} shows this result is no longer true if the weak decomposition property is removed from the assumptions. 

We point out that one can optimize the constant $c_2$ in the John-Nirenberg property, making it as large as possible, by letting $\alpha = Ke$, in order to obtain $c_1 = \sqrt{e}$ and $c_2 = (2kKe)^{-1}$. Here $k$ and $K$ are from Lemma \ref{calderonzygmund}.  

As a side note, if $f \in \Lone(\X)$ and $\cG$ satisfies the finiteness and doubling properties, with doubling constants $a$ and $b$, then the maximal function $M_{\cG}f(x) = \sup \left\{ \fint_{G(x)}\! |f| \, d\mu: G(x) \in \cG \right\}$ is in weak-$\Lone(\X)$, i.e.
$$
\mu(\{x \in X: Mf(x) > \alpha\}) \leq \frac{b}{\alpha}\|f\|_{\Lone(\X)}\,.
$$
The proof is similar to the one found in \cite[Theorem 2.2]{heinonen}.

\subsection{Local Denjoy families}

Observing that all the proofs leading to the John-Nirenberg inequality in Theorem \ref{Denjoytheorem} are local in nature and depend only on a subcollection of $\cG$, we  define local Denjoy families:

\begin{definition}\label{localDenjoy}
For $1<a \leq b$, we say that $\cG_0=\cG(G_0)\subset \{G \in \cG: G \cap G_0 \neq \emptyset \; \& \; \mu(G) \leq a\mu(G_0)\}$ is a local $(a,b)$-Denjoy family for the generator $G_0$ if it satisfies the following conditions:
\vskip 1\baselineskip
\begin{itemize}
\item[] {\bf Shrinking  Property}\\ For almost every $x \in G_0$, there exists a sequence $\{G_n\}\subset \cG_0$ such that $G_n \to x$.\\
\item[] {\bf Denjoy Doubling Property}\\ Every $G \in \cG_0$ with $\mu(G)<\frac{a}{b}\mu(G_0)$ has an $(a,b)$-Denjoy double $G' \in \cG_0$ with respect to $\cG_0$, i.e. the set map $\Omega_a$ is restricted to $\cG_0$.\\
\item[] {\bf Growth Property}\\ For any $G \in \cG_0$, there exists $j \in \N$ such that $\mu^{\ast}\left(\Omega_a^{(j)}(G)\right) \geq \frac{a}{b} \mu(G_0)$.\\
\item[] {\bf Weak Differentiation Property}\\ For every $f \in \Loneloc(\X,\cG_0)$, for almost every $x \in G_0$, $|f(x)| \leq \limsup\limits_{G \to x} \displaystyle\fint_G\! |f| \, d\mu$, where the limsup is taken over $\{G\}\subset\cG_0$ such that $G \to x$. \\
\item[] {\bf Engulfing Property}\\ For any finite subcollection $\{G_n\}_{n=1}^N \subset \cG_0$, there exists $\widetilde{G} \in \cG$ such that $\bigcup\limits_{n=0}^N G_n \subset \widetilde{G}$ and $\mu(\widetilde{G}) \leq b\mu(G_0)$.
\end{itemize}
\end{definition}

Properties 1 to 4 are simply local versions of the properties in Definition \ref{defdenjoy}. The engulfing property was implied in the ``global" definition, by setting $\widetilde{G}=G_0'$ for any finite subcollection. Note that we do not require that $\widetilde{G}$ be in $\cG_0$.

\begin{theorem}
If $\cG(G_0)=\cG_0 \subset \cG$ is a local $(a,b)$-Denjoy family for $G_0 \in \cG$, then we have a John-Nirenberg inequality with respect to $G_0$ for any $f \in \BMO$.
The constants are $c_1=\sqrt{e}$ and $c_2=(2kKe)^{-1}$, where $k=\frac{b^3}{a} + b$ and $K=b^2(b+1)$.
\end{theorem}

\begin{proof}
It suffices to show that the Lemmas \ref{calderonzygmund} and \ref{almostJN} hold for a fixed $G_0$ with $\cG_0$ in place of $\cG$. Then, the result follows from the proof of Theorem \ref{Denjoytheorem}, with the constant $k$ being changed because of the change in the definition of $\mathscr{A}_0$ below in \eqref{newA0}. 

First, we point out that the engulfing property tells us that for any countable pairwise disjoint subcollection of $\cG_0$, the measure of the union is finite. Let $\{G_n\}_{n=1}^{\infty} \subset \cG_0$ and fix for every $N \in \N$ an engulfing set $\widetilde{G}_N$, then
\begin{equation}\label{limit}
\sum\limits_{n=1}^{\infty} \mu(G_n) = \lim\limits_{N \to \infty} \sum\limits_{n=1}^N \mu(G_n) \leq \sup\limits_{N \in \N} \mu(\widetilde{G}_N) \leq b \mu(G_0).
\end{equation}
In turn, this tells us that $\mu(G_n) \to 0$, which is needed in the proof of Lemma \ref{calderonzygmund}. It also allows us to obtain the estimate \eqref{iii} by replacing $G_0'$ with $\widetilde{G}_N$ for finite sums and then taking the limit, as done in \eqref{limit}.

Second, the engulfing property allows us to obtain the estimate \eqref{control} by replacing $G_0'$ with an engulfing $\widetilde{G} \supset G_0 \cup G$ in \eqref{notA1}.

One last modification is needed in order for the proofs of Lemmas \ref{calderonzygmund} and \ref{almostJN} to hold for $G_0$ with its local Denjoy family $\cG_0$. We will redefine $\mathscr{A}_0$ as
\begin{equation}\label{newA0}
\mathscr{A}_0 = \{G \in \cG_0: \mu(G) < \frac{a}{b^2}\mu(G_0)\}\,.
\end{equation}
Clearly, we want to restrict $\mathscr{A}_0$ to $\cG_0$, and the reason we need $b$ to be squared is because when we repeat the decomposition of Lemma \ref{calderonzygmund} for the sequence $\{G_j\}$ first obtained by this lemma, we need to know that for every $j$, the sets in $\{G \in \cG_0: \mu(G) < \frac{a}{b^2}\mu(G_j)\}$ have a double. This follows from the fact that each $G_j$ is the double of a set in $\mathscr{A}_0$, so we know that $\mu(G_j) < \frac{a}{b}\mu(G_0)$, i.e. $G_j$ itself has a double, which will also serve as an engulfing set. Thus we can repeat the decomposition indefinitely.

Finally, note that if $\bigcup\limits_{G \in \cG_0} G \subset G_0$, then $G_0$ could serve as an engulfing set and this would optimize the constants obtained in \eqref{notA1} and \eqref{iii}, giving us $k=\frac{b}{a}$ and $K=b$.
\end{proof}

Note that the result here is local as well, as we only get the John-Nirenberg inequality for a fixed $G_0$, but we do not know anything about the other generators. Thus, we will look at how local Denjoy families can give us the JNP, which is a ``global" property. This will allow us to prove the JNP for rectangles in $\R^n$ (see Example \ref{rectangles}).

\begin{proposition}\label{unionDenjoy}
If every $G \in \cG$ has a local Denjoy family with doubling constants $a_G$ and $b_G$ such that $b=\sup\limits_{G \in \cG} b_G < \infty$, then $\BMO$ has the JNP.

If $\cG = \bigcup\limits_{\omega} \cG_{\omega}$, where $\cG_{\omega}$ is a Denjoy family for every $\omega$ with doubling constants $a_{\omega}$ and $b_{\omega}$ such that $b=\sup\limits_{\omega} b_{\omega}<\infty$, then $\BMO$ has the JNP.
\end{proposition}

\begin{proof}
Note that since every $G\in\cG$ has a local Denjoy family, we do not need to redefine $\mathscr{A}_0$ as in \eqref{newA0}. Applying Lemma \ref{almostJN} for a fixed $G\in\cG$ with $k_G=\frac{b_G^2}{a_G} +b_G$ and $K_G=b^2_G(b_G+1)$, and using $k=b^2+b\geq k_G$ and $K=b^2(b+1)\geq K_G$ since $a_G>1$ and $b\geq b_G$, we have that, for every $N\in\N$,
 $$
\mu\left(E_{Nk\alpha}\right) \leq \mu\left(E_{N k_G\alpha}\right) \leq \left(\frac{K_G}{\alpha} \right)^N \mu(G) \leq \left(\frac{K}{\alpha} \right)^N \mu(G)\,.
$$
Thus, \eqref{keyinequality} holds for $k$ and $K$, which implies Theorem \ref{Denjoytheorem}, completing the proof.

The proof of the second statement follows in a similar manner.
\end{proof}

\subsection{The shrinking property}
In this section, we ask which properties on $(\X,\cG)$ are necessary for $\BMO$ to possess the JNP. In particular, we show that the shrinking property is one of them. 

\begin{proposition}\label{notJNP}
If there exists a pairwise disjoint collection of sets $\{E_n \}_{n=1}^{\infty} \subset \cM^*$ with $E = \bigcup\limits_{n=1}^{\infty} E_n$ such that
\begin{itemize}
\item $\mu(E) < \infty$,
\item $\alpha = \inf \left\{ \mu(G): \mu(G \cap E)>0 \right\} >0$, and
\item there exists $G_0 \in \cG$ such that it contains infinitely many elements of $\{E_n\}_{n=1}^{\infty}$,
\end{itemize}
then $\BMO$ does not have the JNP.
\end{proposition}

\begin{proof}
First, if $\widetilde{E}_n = E_n \cap G_0$ and $\widetilde{E} = \bigcup\limits_{n=1}^{\infty} \widetilde{E}_n$, then $\{ \widetilde{E}_n \}_{n=1}^{\infty}$ has an infinite subsequence that satisfies the same conditions as above. Therefore, we can assume without loss of generality that $E \subset G_0$.
Second, since $\sum\limits_{n=1}^{\infty} \mu(E_n) = \mu(E) < \infty$, then $\mu(E_n) \to 0$. Therefore, we can assume without loss of generality that $\mu(E_n)< 2^{-(n^2+1)}$ for every $n \in \N$.

Let $\displaystyle{f = \sum\limits_{n=1}^{\infty} \frac{2^{-n}}{\mu(E_n)} \chi_{E_n}}$. If $\mu(G \cap E)>0$, then setting $M=\alpha^{-1}$,
$$
\fint_G\! |f-f_G| \, d\mu \leq \frac{2}{\mu(G)} \int_G \!|f| \, d\mu \leq 2M \, \sum\limits_{n=1}^{\infty} \int_{E_n}\! |f| \, d\mu = 2M \, \sum\limits_{n=1}^{\infty} 2^{-n} = 2M\,.
$$
If $\mu(G \cap E) = 0$, then $f=0$ almost everywhere on G, so $\displaystyle\fint_G\! |f-f_G| \, d\mu = 0$. Therefore, $f \in \BMO$, but $f \notin L^2_{loc}(\X,\cG)$ since
$$
\int_{G_0}\! |f|^2 \, d\mu = \sum\limits_{n=1}^{\infty} \frac{2^{-2n}}{\mu(E_n)}\geq \sum\limits_{n=1}^{\infty} \frac{2^{-2n}}{2^{-(n^2 +1)}}= \sum\limits_{n=1}^{\infty} 2^{(n - 1)^2} = \infty\,.
$$
Hence, by Corollary \ref{JNPlploc}, $\BMO$ does not have the JNP.
\end{proof}

\begin{lemma}\label{countablecover}
Let $\mathcal{F} \subset \cM^*$ and $A \in \cM^*$. If $\mathcal{F}$ does not admit a countable essential subcover of $A$, then there exists $E \in \cM^*$ such that $E \subset A$ and $\mu(F \cap E) = 0$ for every $F \in \mathcal{F}$. 
\end{lemma}

\begin{proof}
Write $\mathscr{A} = \{ F \in \mathcal{F} : \mu  ( F \cap A ) > 0 \}$. First, if $\mathscr{A} = \emptyset$, then it cannot essentially countably cover $A$, since this implies that $\mu(F \cap A) = 0$ for every $F \in \mathcal{F}$, but $\mu(A)>0$. Then $E = A$ and the proof is done.

Otherwise, as in the second part of the proof of Lemma \ref{normFCPsigma2}, we can get a countable union of elements of $\{F \cap A: \mu(F \cap A)>0\}$, denoted $Q$, that is maximal in the sense that $\mu(F \cap (A \setminus Q)) = 0$ for any $F \in \mathcal{F}$ and $\mu(A \setminus Q) > 0$.
\end{proof}

\begin{lemma}\label{Geps}
Writing $\mathscr{G}_{\eps} = \left\{ G \in \cG: \mu(G) < \eps \right\}$, the following statements are equivalent:
\begin{enumerate}
\item[(i)] for every $A \in \cM^*$ and $\eps>0$, there exists $G \in \cG_{\eps}$ such that $\mu(G \cap A)>0$;
\item[(ii)] for every $A \in \cM^*$ and $\eps>0$, $\cG_{\eps}$ admits a countable essential subcover of $A$.
\end{enumerate}
\end{lemma}

\begin{proof}
To show that (i) implies (ii), fix $A \in \cM^*$ and $\eps>0$. Let $\mathcal{F} = \mathscr{G}_{\eps}$ and $\mathscr{A} = \{ G \in \mathscr{G}_{\eps}: \mu(G \cap A)>0 \}$. If $\mathscr{A}$ does not admit a countable essential subcover of A, then by Lemma \ref{countablecover}, there exists $E \in \cM^*$ such that $\mu(G \cap E) = 0$ for every $G \in \mathscr{G}_{\eps}$, which is a contradiction. Therefore, $\mathscr{G}_{\eps}$ is a countable essential cover of $A$ for every $A \in \cM^*$ and for every $\eps>0$

To show that (ii) implies (i), if we fix a countable essential subcover of $\mathscr{G}_{\eps}$ for $A$ (for an arbitrary $\eps>0$ and $A \in \cM^*$), then there must be at least one $G \in \mathscr{G}_{\eps}$ such that $\mu(G \cap A) >0$ since $\mu(A)>0$.
\end{proof}

\begin{theorem}\label{epscover}
Let $(\X,\cG)$ be a nonatomic, complete decomposable measure space endowed with a collection of generators such that $\BMO$ is a Banach space modulo constants. If $\BMO$ has the JNP, then $\mathscr{G}_{\eps}$ admits a countable essential subcover of $A$ for every $A \in \cM^*$ and for every $\eps>0$.
\end{theorem}

\begin{proof}
Assume that for some $A \in \cM^*$ and some $\delta>0$, $\cG_{\delta}$ does not countably essentially cover $A$. Then, by the contrapositive of Lemma \ref{Geps}, there exists $\widetilde{E} \in \cM^*$ and $\eps>0$ such that for every $G \in \mathscr{G}_{\eps}$, $\mu(G \cap \widetilde{E}) = 0$. In other words, if $G \in \cG$ and $\mu(G \cap \widetilde{E}) >0$, then $\mu(G) \geq \eps$.

As $\BMO$ is a Banach space, it follows that $(\X,\cG)$ is $\sigma$-decomposable by Theorem \ref{completeness}, so we know there exists a $G_0 \in \cG$ such that $\mu(G_0 \cap \widetilde{E}) >0$. By 215D in \cite{fremlin2}, we can fix a disjoint collection of sets $\{E_n \}_{n=1}^{\infty} \subset \cM^*$ such that for every $n \in \N$, $\mu(E_n)>0$ and $E = \left( \bigcup\limits_{n=1}^{\infty} E_n \right) \subset (G_0 \cap \widetilde{E}) \subset G_0$. Hence, $E \in \cM^*$ and, if $\mu(G \cap E)>0$, then $\mu(G \cap \widetilde{E})>0$, so $\mu(G) \geq \eps$. Therefore, we have all the necessary conditions of Proposition \ref{notJNP}, and so $\BMO$ does not have the JNP.
\end{proof}

\begin{corollary}\label{shrinkingcorollary}
Under the conditions of Theorem \ref{epscover}, if $\BMO$ has the JNP, then for almost every $x \in X$ there exists a sequence $\{G_n\} \subset \cG(x)$ such that $G_n \to x$, i.e. $\cG$ has the shrinking property.
\end{corollary}

\begin{proof}
Fix a decomposition $\{X_{\omega}\}_{\omega \in \Omega}$ of $X$. Now, fix some $\omega \in \Omega$. Then, by Theorem \ref{epscover}, for every $n \in \N$ we can find an essential countable cover of $X_{\omega}$ by sets of measure smaller than $1/n$.

Therefore, almost every $x \in X_{\omega}$ is contained in the intersection over all $n$ of these covers. Since it is true for every $\omega \in \Omega$, this completes the proof.
\end{proof}


\section{Examples}

\subsection{Topological local integrability}

Here we will look at the definition of local integrability given by Definition \ref{deflploc} versus the classical definition, when $(\X,\tau)$ is a topological measure space (see Definition \ref{toplocint}) and $\cG$ is an open cover.

\begin{definition}
We say that $f \in \Lploc(\X, \tau)$ if for every compact set $K \subset X$,
$$
\int\limits_K\! |f|^p \, d\mu < \infty\,.
$$
\end{definition}

If we assume that $\cG$ is an open cover, one can notice, as shown in \cite{dafni1}, that $\BMOp \subset \Lploc(\X, \cG) \subset \Lploc(\X, \tau)$. If, in addition, we assume that every $G \in \cG$ is contained in a compact set, then we can equivalently define $\BMOp$ as
\begin{equation}\label{classicbmo}
\BMOp = \{f \in L^p_{loc}(\X,\tau): \|f\|_{\BMOp} < \infty\}.
\end{equation}
Therefore, the BMO space as defined in Definition \ref{defbmo} includes the most common BMO spaces, defined over cubes or rectangles. 

Conversely, the following example will show that even when considering $\cG$ to be the collection of open balls, if the space is not locally compact, then $\Loneloc(\X,\cG)$ may be a proper subset of $\Loneloc(\X,\tau)$, hence both definitions of BMO may not be equivalent. We will build a Lindel\"{o}f, $\sigma$-compact metric space such that $\BMO$ is a Banach space, but $X$ is not locally compact.

When $(X,\rho)$ is a metric space, we use the notation $B(x,r)$ for the open ball with fixed centre $x\in{X}$ and radius $r>0$.

\begin{example}\label{localintegrability}
Let $X = \R^2$ and the measure be Lebesgue measure with the Lebesgue $\sigma$-algebra. We will consider a point $x \in X$ with its polar coordinates, $x = (r_x,\theta_x)$, where $r_x=|x|$ and $\theta_x \in [0,2\pi)$. First, define the sequence $\theta_k = \sum\limits_{i=0}^{k-1} \frac{\pi}{2^i}$ for $k\geq 1$ and $\theta_0 = 0$. Then we will partition the space into subsets $\{ P_k \}_{k=1}^{\infty}$, where 
$$
P_k = \{ x = (r,\theta) \in X \backslash \{(0,0)\}: \theta \in [\theta_{k-1},\theta_k) \} \; \text{and} \; P_0 = \{(0,0)\}\,.
$$
What we are doing here is cutting up $\R^2$ like a pie, but with the pieces getting infinitely smaller.

Now we define the following metric:
$$
\rho(x,y) =
\begin{cases}
|x-y|, & \text{if $x$ and $y$ belong to the same $P_k$} \\
|x| + |y|, & \text{otherwise}. \\
\end{cases}
$$
One can verify geometrically that any ball in $(\R^2,\rho)$ is a Borel set in the Euclidean metric, hence the Lebesgue measure is a Borel measure on this space, as well. Let $\cG$ be the collection of open balls in $(\R^2,\rho)$. By the construction of the metric, for any $x \in P_k$, $B(x,|x|) \subset P_k$. Hence, each $P_k$ is open in this metric and one can verify that each $P_k$ is Lindel\"{o}f and connected.

Since the partition is countable, it follows that $(\R^2,\rho)$ is Lindel\"{o}f. And, since the balls centred at the origin are all identical to their Euclidean counterpart, then we get that the whole space $\R^2$ is connected. Therefore, by Theorem \ref{topologytheorem}, $\BMO$ is a Banach space.

The interesting aspect of this space is that it is $\sigma$-finite and Lindel\"{o}f, but not locally compact, because of the point $(0,0)$. Let $D_{\eps}$ be a closed ball centred at $(0,0)$ of radius $\eps$ and $B_{\delta}$ and $B_{\lambda}$ be open balls centred at $(0,0)$ of radius $\delta > \eps$ and $\lambda < \eps$, respectively. Then we get that $\{(P_k \cap B_{\delta}) \}_{k=1}^{\infty}$ with $B_{\lambda}$ is an open cover of $D_{\eps}$ that does not admit a finite subcover. Since $\eps$ is arbitrary, we get that $(0,0)$ has no neighbourhood with compact closure, and so $X$ is not locally compact.

Therefore if we let $\cG$ be the collection of open balls, then for any $x \in \R^2$ and $r>|x|$, the ball $B(x,r)$ contains a ball centred at $0$, hence it is not contained in a compact set. This means that any compact set is the finite union of compact sets contained in different $P_k$'s.

Let $B_k = B(0,1) \cap P_k$ and $f=\sum\limits_{k=1}^{\infty} \frac{k}{\mu(B_k)}\chi_{B_k}$. Then $f \in \Loneloc(\X,\tau) \setminus \Loneloc(\X,\cG)$, since 
$$
\int_{B(0,1)}\! |f| \, d\mu = \sum\limits_{k=1}^{\infty} k = \infty\,,
$$
but for any $K$ compact, there exists $N$ such that 
$$
\int_K\! |f| \, d\mu \leq \sum\limits_{k=1}^N \int_{B_k}\! |f| \, d\mu = \sum_{k=1}^{N}k < \infty\,.
$$
\end{example}

\subsection{Examples of decomposable spaces that are not $\sigma$-finite}

Here, we construct a decomposable measure space $\X$ that is not $\sigma$-finite, coupled with a cover $\cG$ that makes $\BMO$ a Banach space. Then, we will add a metric, making $(\X,\tau)$ a topological measure space and seeing how it relates with Theorem \ref{topologytheorem} and Corollary \ref{metriccorollary}.

\begin{example}\label{ex2}
Consider $X=\R^2$. We will consider points in polar coordinates. We start by partitioning $X$ into $\{ M_{\theta} \}_{\theta \in [0,2\pi)}$, where $M_{\theta} = \{ x \in X \setminus \{(0,0)\}: \theta_x = \theta \}$ for every $\theta \in [0,2\pi)$ and $O = \{(0,0)\}$.

For $\theta \in [0,2\pi)$, if $m$ is the one-dimensional Lebesgue measure and $\varphi_{\theta}$ the isomorphism from $M_{\theta}$ to $(0,\infty)$, then define $\mu_{\theta} = m \circ \varphi_{\theta}$. If we let $\mu(O)=0$ and $\mu(A) = \sum\limits_{\theta \in [0,2\pi)} \mu_{\theta}(A\cap M_{\theta})$, then we have a well-defined measure space $\X$ that is, by its construction, decomposable but not $\sigma$-finite. We can visualize each $M_{\theta}$ as a ray that starts at the origin. 

We now define a metric on $\R^2$ that will be somewhat ``compatible" with $\X$. For every $x,y \in X$, let
$$
  \rho(x,y) =
  \begin{cases}
    |x - y|, & \text{if $\theta_x = \theta_y$} \\
     |x| + |y|, & \text{if $\theta_x \neq \theta_y $} \\
  \end{cases}
$$
This is also known as the radial or hedgehog metric (see \cite[p. 251]{engel}). Note that we calculate the distance between two points by moving along the rays (the hedgehog's spikes) and passing through the origin (the hedgehog's body), if two points are not on the same ray.

To obtain a cover $\cG$ such that $\BMO$ is a Banach space, let $G_{I}^{\theta} = \varphi_{\theta}^{-1}(I)$ for every interval $I \subset (0,\infty)$ and let
\begin{itemize}
\item $\widetilde{\cG} = \{ O \cup G^0_{(0,1]} \cup G^{\theta}_{(0,1]} \}_ {\theta \in (0,2\pi)}$,\\
\item $\cG_{\theta} = \{G^{\theta}_I: \, \text{{\it I is any closed interval contained in}} \, (0,\infty) \}$, and \\
\item $\cG = \widetilde{\cG} \cup \left( \bigcup\limits_{\theta \in [0,2\pi)} \cG_{\theta} \right)$.
\end{itemize}

By the properties of the intervals in $\R$, we know that for each $\theta \in [0,2\pi)$, $\cG_{\theta}$ has the FCP, and since $\widetilde{\cG}$ has the FCP by construction, using transitivity, we get that $\cG$ has the FCP. Also each $M_{\theta}$ is countably covered by closed intervals, so $(\X,\cG)$ is $\sigma$-decomposable. By Theorem \ref{completeness}, this is enough to show that $\BMO$ is a Banach space.

Notice that every $G \in \cG$ is a closed set, every ball has strictly positive measure and the space is connected. Therefore the assumption that $\cG$ is an open cover in Theorem \ref{topologytheorem} is not a necessary conditions in order for $\BMO$ to be a Banach space.
\end{example}

Note that any ball centred at the origin is just a regular two-dimensional Euclidean ball. Then any open set covering the point $(0,0)$ will have measure infinity, since any ball centred at the origin is an uncountable union of sets of measure greater than 0. As such, it is impossible to cover $X$ completely with an open cover $\cG$, which also follows from Corollary \ref{metriccorollary}. 

\begin{example}\label{ex3}
Consider the measure space $\X$ and metric $\rho$ as in Example \ref{ex2}, but remove $O = \{(0,0)\}$ from $X$. Let 
\begin{itemize}
\item $\cG_{\theta} = \{G^{\theta}_I: \, \text{{\it I is any open interval in}} \, (0,\infty) \}$,\\
\item $\widetilde{\cG} = \{ (G^{\theta_1}_{I_1} \cup G^{\theta_2}_{I_2}): \, \theta_1, \theta_2 \in [0,2\pi) \; I_1 \; \text{{\it and }} \; I_2 \; \text{{\it are any open intervals in}} \, (0,\infty) \}$, and\\
\item $\cG = \widetilde{\cG} \cup \left( \bigcup\limits_{\theta \in [0,2\pi)} \cG_{\theta} \right)$.
\end{itemize}
Then we can use the same arguments as in Example \ref{ex2} to show that for this $(\X,\cG)$, $\BMO$ is a Banach space.
\end{example}

Since in the example above, $X$ is neither connected nor Lindel\"{o}f, $\cG$ is an open cover, and every ball has strictly positive measure, this example shows that the assumption that $X$ is connected in Corollary \ref{metriccorollary} is not necessary. Furthermore, we see that $(X,\rho)$ may not be Lindel\"{o}f when $X$ is not connected.

\subsection{Denjoy families}

In this subsection, we consider examples related to Denjoy families. 

\begin{example}\label{rectangles}
Let $\X$ be $\R^n$ with the Lebesgue measure and $\cG$ be the collection of open rectangles. Classically, as in \cite{kor}, the collection of rectangles only includes those with sides parallel to the axes, but we will here consider all possible rectangles.

Fix $R \in \cG$ and let $\cG(R) \subset \cG$ be the collection of rectangles with sides parallel to $R$ and with the same proportions between sidelengths as $R$. In other words, $\cG(R)$ is the collection of translated homothetic transformations of $R$, with the homothecy centre at the centre of $R$. We will denote by $cR+x$ the homothetic transformation of $R$ of scale $c$, translated by $x \in \R^n$.

We will show that $\cG(R)$ is a Denjoy family. Note that by properties of the Lebesgue measure, $\mu(cR+x)=\mu(cR)=c^n \mu(R)$. Therefore, we can see geometrically that if we let $a=(3/2)^n$, then $\Omega_a(R) \subset 4R$, with $\mu(4R) = 4^n \mu(R)$. So with $b=4^n$, and since for any $R_1 \in \cG(R)$, $\cG(R_1) = \cG(R)$, we get that $\cG(R)$ is a Denjoy family. Since $\cG = \bigcup\limits_{R \in \cG} \cG(R)$ and we can pick the uniform constants $a=(3/2)^n$ and $b=4^n$, we can apply by Proposition \ref{unionDenjoy} to conclude that $\BMO$ has the JNP.
\end{example} 

Note that the constant in the exponent for the John-Nirenberg inequality on rectangles obtained by Korenovskii ($c_2=\frac{2}{e}$) in \cite[Theorem 3.23]{kor} provides a better estimate than one that can obtained from Lemma \ref{calderonzygmund} and Theorem \ref{Denjoytheorem}. In fact, in Example \ref{rectangles} we could take any $a>1$, hence any $b>3^n$. Therefore, we could take the limit $b \to 3^n$ in \eqref{keyinequality} to get the best constant obtainable by our tools: $$c_2=\frac{1}{2(3^{3n})(3^n+1)e} < \frac{2}{e}\,,$$
meaning that our estimate is bigger than the one of Korenovskii. This is not surprising though, since our proof relies only on measure-theoretic tools.

Now we come to a large and important important class of measure spaces equipped with a natural Denjoy family: the spaces of homogeneous type \`{a} la Coifman-Weiss \cite{coifweiss}. Before proving this assertion, we present the relevant definitions. 

Let $(X,d)$ be a quasimetric space, i.e. $d:X \times X \to [0,\infty)$ has all the properties of a metric, except that the triangle inequality is satisfied with a constant $K\geq 1$:
\begin{equation}\label{triangle}
d(x,y) \leq K (d(x,z)+d(z,y))\,.
\end{equation}
We will require that the balls $B(x,r)=\{y \in X: d(x,y)<r\}$ form an open basis for the topology $\tau$, i.e. for any $x \in X$, $r>0$ and $y \in B(x,r)$, there exists $\eps_y$ such that $B(y,\eps_y) \subset B(x,r)$. Let $\cB$ be the collection of all open balls. Note that in some cases the balls are not open, but as shown in \cite{macias}, in this case there is an equivalent quasimetric for which the balls are open, and so with the doubling property, we would get equivalent BMO spaces with either quasimetric. 

Now let $\mu$ be a nonatomic positive Borel measure on $(X,d)$. We will require that $\mu$ satisfies the following doubling property: there exists a constant $A > 1$ such that for every $B(x,r) \in \cB$,
\begin{equation}\label{doubling}
0<\mu(B(x,2Kr)) \leq A \mu(B(x,r))<\infty.
\end{equation}
As we will see, our definition of Denjoy doubling is a generalization of this doubling property, without any notion of distance. Note that the constant $K$ in \eqref{doubling} is the same as in \eqref{triangle}. Finally, we will also require that the measure is outer regular, i.e. for every $E \in \cM$, 
$$
\mu(E) = \inf \left\{ \mu(U) : U \; \text{is open}, E \subset U \right\}\,.
$$
Under all these assumptions, we say that $(\X,d)$ is a space of homogeneous type. 

Note that when $K=1$, $d$ is a metric. The outer regularity property is often assumed for doubling metric measure spaces, such as in \cite{heinonen}. Also, by Theorem 1.3 in \cite{coifweiss}, any bounded open set is a countable union of balls, hence so is any open set $U$, since for every $n \in \N$, $U = \bigcup\limits_{n=1}^{\infty} U \cap B(x,n)$, where $x \in U$. Therefore, the Borel $\sigma$-algebra is the same as the one generated by balls.

\begin{lemma}
Let $(\X,d)$ be a space of homogeneous type. For any $B_0=B(x,r)\in \cB$, if $\mu(B_0)<A^{-N}\mu(X)$ for some $N \in \N$, then there exist $R_n \geq r$ such that $$A^{n-1}\mu(B_0)\leq \mu(B(x,R_n)) < A^n \mu(B_0)\,$$
for any $1 \leq n \leq N$.
\end{lemma}

This means that although the measure of balls can have jumps, it cannot jump more than the doubling constant.

\begin{proof}
Fix a ball $B_0=B(x,r)$ such that $\mu(B_0)<A^{-N}\mu(X)$. First note that since $X = \bigcup\limits_{i=1}^{\infty} B(x,i)$, then for any $2 \leq n \leq N$, there exists $s_n > r$ such that $\mu(B(x,s_n)) \geq A^n \mu(B_0)$. Fix $n$ and let $q_n$ be the infimum over all such $s_n>r$.

Then for every $\eps>0$, $\mu(B(x,q_n-\eps)) < A^n\mu(B_0) \leq \mu(B(x,q_n+\eps)$. Let $\eps_n = \frac{2K-1}{2K+1}q_n$, then $2K(q_n-\eps_n)=(q_n+\eps_n)$ and by \eqref{doubling}, $$\mu(B(x,q_n-\eps_n))\geq A^{-1}\mu(B(x,(2K(q_n-\eps_n))) = A^{-1}\mu(B(x,q_n+\eps_n)) \geq A^{n-1}\mu(B_0)\,.$$ So letting $R_n=q_n-\eps_n$, we have that $\mu(B(x,R_n)) >\mu(B_0)$, so $R_n > r$ and we have proven the claim.
\end{proof}

\begin{example}\label{homogeneous}
Let $(\X,d)$ be a space of homogeneous type. We will show that the collection $\cB$ satisfies the conditions of Proposition \ref{unionDenjoy}, demonstrating that $\mathrm{BMO}_{\cB}(\X)$ has the JNP.

Since $\mu$ is nonatomic by assumption, $B(x,r) \subset B(x,R)$ whenever $r \leq R$ and $\{x\} = \bigcap\limits_{r>0} B(x,r)$, hence $\lim\limits_{r \to 0} \mu(B(x,r)) = \mu(\{x\}) = 0$, so $\cB$ has the shrinking property. 

Fix $\widetilde{B} \in \cB$ and let $a = A$, with $A$ being the constant in \eqref{doubling}. We want to show that with $b=A^8$, $\mathscr{D}_a= \{B \in \cB: B \cap \widetilde{B} \neq \emptyset \; \& \; \mu(B)<a\mu(\widetilde{B})\}$ is a local Denjoy family for $\widetilde{B}$. Take an arbitrary $B(x,r)=B_0 \in \mathscr{D}_b= \{B \in \cB: B \cap \widetilde{B} \neq \emptyset \; \& \; \mu(B)<\frac{a}{b}\mu(\widetilde{B})\}$. 

By the previous claim, using $N=7$, there exists $R>r$ such that 
$$
A^3 \mu(B_0) \leq \mu(B(x,R))<A^4\mu(B_0)\,.
$$
Fix such an $R$. Now we want to show that $\Omega_{a=A}(B_0) \subset B(x,(2K)^4R)$, which would prove the doubling property, since $$\mu(B(x,(2K)^4R)) \leq A^4\mu(B(x,R))<A^8\mu(B(x,r))\,.$$

We first start by bounding the radius of the balls that form $\Omega_a(B_0)$. Take $y \in X$ and $r_y > R+r$. If $B(y,r_y) \cap B_0 \neq \emptyset$, then $d(x,y) < K(r_y+r)$, so for every $z \in B(x,R)$, 
$$d(y,z) < K(K(r_y + r) + R) < (2K)^2r_y\,.$$
Therefore, $B(x,R) \subset B(y,(2K)^2r_y)$, so $\mu(B(y,r_y)) > A^{-2}\mu(B(x,R)) \geq A \mu(B_0)$. Hence, we must have $r_y \leq r + R$, so $\Omega_a(B_0) \subset B(x,(2K)^3(R+r)) \subset B(x,(2K)^4R)$. Note that by choosing $a=A$, then for any $x \in X$ and $r>0$,  $B(x,2r) \subset \Omega_a \left(B(x,r) \right)$. This shows that $\mathscr{D}_a$ possesses the growth property.

As for the weak differentiation property, because we assumed outer regularity of the measure, then for any $f \in \Loneloc(\X,\cB)$, the Lebesgue differentiation theorem holds at almost every $x \in X$ (see \cite{toledano}). Thus, this implies the weak differentiation property.

Finally, if $\mu(\widetilde{B})<A^{-8}\mu(X)$, then $\widetilde{B} \left(x,(2K)^4R \right)$, as obtained previously, satisfies the engulfing property for all finite subcollections of $\mathscr{D}_a$. Otherwise, by the properties of a quasimetric, for any finite collection of balls, we can find a ball $B'$ that contains it and $\widetilde{B}$, which would satisfy the engulfing property since $\mu(\widetilde{B}) \geq A^{-8}\mu(X)$ implies that $\mu(B') \leq \mu(X) \leq A^8 \mu(\widetilde{G})$.

Therefore $\cD_a$ is a local Denjoy family for $\widetilde{B}$ with doubling constants $a=A$ and $b=A^8$. Since the constants are uniform over all of $\cB$, we have all the conditions for Proposition \ref{unionDenjoy}.
\end{example}

\begin{example}\label{ex3.2}
Let $(\X,\cG)$ be the space defined in Example \ref{ex3}. We can see that $\cG$ is a union of Denjoy families. In fact each $\cG^{\theta}$ is a Denjoy family by properties of intervals and so are any $(\cG^{\theta_1} \cup \cG^{\theta_2})$. Since all $\cG^{\theta}$'s are isomorphic, we can pick the doubling constants to be uniform. Therefore, by Proposition \ref{unionDenjoy}, this BMO space has the JNP, even though it is not $\sigma$-finite.
\end{example}

\subsection{The weak convergence property}

Now we consider an example which shows the weak differentiation property cannot be completely omitted. That is, while the shrinking property is necessary for the JNP in the nonatomic case, the shrinking, doubling, and growth properties alone are not sufficient. 

\begin{example}\label{noJNPexample}
Let $\X$ be $(0,1)^2$ equipped with two-dimensional Lebesgue measure and the Lebesgue $\sigma$-algebra, with generators from the collection $\cG = \{R^d_c = (c,d) \times (0,1): 0 \leq c<d \leq 1 \}$.

As $\X$ is of finite measure, thus decomposable, we can apply Theorem \ref{completeness} to conclude that $\BMO$ is a Banach space modulo constants, since the fact that $R^1_0=(0,1)^2=X \in \cG$ implies that $(\X,\cG)$ is $\sigma$-decomposable, and $\cG$ has the FCP.

Also, for every $x=(x_1,x_2) \in X$, there exists $N \in \N$ such that $\{ R_n = (x_1-1/n,x_1+1/n) \times (0,1) \}_{n>N} \subset \cG$ and $\mu(R_n) \to 0$. Therefore $\cG$ has the shrinking property.

Let $f(x,y) = y^{-1/2}$. For any $R \in \cG$, 
$$\fint_R\! |f| \, d\mu = \frac{1}{b-a} \int^b_a \int^1_0 y^{-1/2} \, dy dx = \int^1_0 y^{-1/2} \, dy = 2\,.$$
This is enough to show that $f \in \BMO$, but
$$\fint_R \!|f|^2 \, d\mu = \int^1_0 y^{-1}\, dy = \infty\,$$ for any $R \in \cG$. By Corollary \ref{JNPlploc}, this is enough to show that $\BMO$ does not have JNP.

Note that if we let $a=2$, for every $R=R_c^d \in \cG$, setting $R'=R_{c'}^{d'}$ with $c'= \max \{(3c-2d),0\}$ and $d'= \min \{(3d-2c),1\}$, we have that $\Omega_a(R) = R'$ and $\mu(R') \leq 5 \mu(R)$. Hence $\cG$ satisfies the doubling property with constants $a=2$ and $b = 5$. We already showed that it satisfies the shrinking property. Since for any $R \in \cG$, $\Omega_a(R)=R'$ and $\mu(R') \geq \min \{1,3\mu(R)\}$, we have that $\cG$ satisfies the growth property: $\mu(\Omega_a^j(R)) \to \mu(X)=1$.

Therefore, $\cG$ satisfies all the properties of a Denjoy family, except for the weak convergence property, since for any $y<1/4$, 
$$|f(x,y)|>2 = \lim\limits_{R \to (x,y)} \fint_R\! |f|\,d\mu\,.$$
\end{example}

In conclusion, the example here shows that we cannot omit the weak differentiation property from the properties of a Denjoy family; it must at least be replaced by some other, possibly weaker, condition.

\section*{{Acknowledgements}}

The authors would like to thank the referees for their careful reading and suggested corrections.


\section*{}


\begin{thebibliography}{99}
\bibitem{mani} Brezis, H.; Nirenberg, L. Degree theory and BMO. I. Compact manifolds without boundaries. {\it Selecta Math. (N.S.)} 1 (1995), no. 2, 197-263.

\bibitem{buckley} Buckley, S. M. Inequalities of John-Nirenberg type in doubling spaces. {\it J. Anal. Math.} 79 (1999), 215-240.

\bibitem{cl} Chen, Y. Z.; Lau, K.-S. Some new classes of Hardy spaces. {\it J. Funct. Anal.} 84 (1989), no. 2, 255-278.

\bibitem{coifweiss} Coifman, R. R.; Weiss, G. Analyse harmonique non-commutative sur certains espaces homog\`{e}nes. (French) \'{E}tude de certaines int\'{e}grales singuli\`{e}res. Lecture Notes in Mathematics, Vol. 242. {\it Springer-Verlag, Berlin-New York}, 1971. v+160 pp.

\bibitem{metric} Coifman, R. R.; Weiss, G. Extensions of Hardy spaces and their use in analysis. {\it Bull. Amer. Math. Soc.} 83 (1977), no. 4, 569-645.

\bibitem{cs}  Cotlar, M.; Sadosky, C. Two distinguished subspaces of product BMO and Nehari-AAK theory for Hankel operators on the torus. {\it Integral Equations Operator Theory} 26 (1996), no. 3, 273-304.

\bibitem{dafni1} Dafni, G.; Gibara, R. BMO on shapes and sharp constants. {\it Advances in harmonic analysis and partial differential equations}, 1-33, Contemp. Math., 748, {\it Amer. Math. Soc., Providence, RI}, 2020.

\bibitem{guzman} de Guzm\'{a}n, M. Differentiation of integrals in $R^n$. With appendices by Antonio C\'{o}rdoba, and Robert Fefferman, and two by Roberto Moriy\'{o}n. Lecture Notes in Mathematics, Vol. 481. {\it Springer-Verlag, Berlin-New York}, 1975. xii+266 pp.

\bibitem{denjoy} Denjoy, A. Une extension du th\'{e}or\`{e}me de Vitali. (French) {\it Amer. J. Math.} 73 (1951), 314-356.

\bibitem{duomo}  Duoandikoetxea, J.; Mart\'{i}n-Reyes, F. J.; Ombrosi, S. On the $A_\infty$ conditions for general bases. {\it Math. Z.} 282 (2016), no. 3-4, 955-972.

\bibitem{engel} Engelking, R. {\it General topology}. Translated from the Polish by the author. Second edition. Sigma Series in Pure Mathematics, 6. {\it Heldermann Verlag, Berlin}, 1989. viii+529 pp.

\bibitem{fremlin1}  Fremlin, D. H. Measure theory. Vol. 1. The irreducible minimum. Corrected third printing of the 2000 original. {\it Torres Fremlin, Colchester}, 2004. 108+5 pp. (errata).

\bibitem{fremlin2} Fremlin, D. H. Measure theory. Vol. 2. Broad foundations. Corrected second printing of the 2001 original. {\it Torres Fremlin, Colchester}, 2003. 563+12 pp. (errata). 

\bibitem{spanish} Garc\'{i}a-Cuerva, J.; Rubio de Francia, J. L. Weighted norm inequalities and related topics. North-Holland Mathematics Studies, 116. Notas de Matem\'{a}tica [Mathematical Notes], 104. {\it North-Holland Publishing Co., Amsterdam}, 1985. x+604 pp.


\bibitem{hy} Hadwin, D.; Yousefi, H. A general view of BMO and VMO. {\it Banach spaces of analytic functions}, 75-91, Contemp. Math., 454, {\it Amer. Math. Soc., Providence, RI}, 2008. 

\bibitem{ht} Hart, J.; Torres, R. H. John-Nirenberg inequalities and weight invariant BMO spaces. {\it J. Geom. Anal.} 29 (2019), no. 2, 1608-1648.

\bibitem{heinonen}  Heinonen, J. Lectures on analysis on metric spaces. Universitext. {\it Springer-Verlag, New York}, 2001. x+140 pp.

\bibitem{jawerth} Jawerth, B. Weighted inequalities for maximal operators: linearization, localization and factorization. {\it Amer. J. Math.} 108 (1986), no. 2, 361-414.

\bibitem{john2}  John, F.; Nirenberg, L. On functions of bounded mean oscillation. {\it Comm. Pure Appl. Math.} 14 (1961), 415-426.

\bibitem{kor} Korenovski\u{\i}, A. A. The Riesz ``rising sun" lemma for several variables, and the John-Nirenberg inequality. (Russian.) {\it Mat. Zametki} 77 (2005), no. 1, 53-66; translation in {\it Math. Notes} 77 (2005), no. 1-2, 48-60. 

\bibitem{kor2} Korenovskii, A. Mean oscillations and equimeasurable rearrangements of functions. Lecture Notes of the Unione Matematica Italiana, 4. {\it Springer, Berlin; UMI, Bologna}, 2007. viii+188 pp.

\bibitem{ly} Lu, S.; Yang, D. The central BMO spaces and Littlewood-Paley operators. {\it Approx. Theory Appl. (N.S.)} 11 (1995), no. 3, 72-94.

\bibitem{macias} Mac\'{i}as, R. A.; Segovia, C. Lipschitz functions on spaces of homogeneous type. {\it Adv. in Math.} 33 (1979), no. 3, 257-270.

\bibitem{perez} P\'{e}rez, C. Weighted norm inequalities for general maximal operators. Conference on Mathematical Analysis (El Escorial, 1989). {\it Publ. Mat.} 35 (1991), no. 1, 169-186.

\bibitem{ns} Nielsen, M.; \v{S}iki\'{c}, H. Muckenhoupt class weight decomposition and BMO distance to bounded functions. {\it Proc. Edinb. Math. Soc.} (2) 62 (2019), no. 4, 1017-1031.

\bibitem{toledano} Toledano, R. A note on the Lebesgue differentiation theorem in spaces of homogeneous type. {\it Real Anal. Exchange} 29 (2003/04), no. 1, 335-339.

\end{thebibliography}
\end{document}